\documentclass[12pt]{amsart}

\usepackage{setspace}
\usepackage{moreverb}

\newcommand{\ignore}[1]{}

\newtheorem{theorem}{Theorem}[section]

\newtheorem{corollary}[theorem]{Corollary}
\newtheorem{proposition}[theorem]{Proposition}

\theoremstyle{definition}

\newtheorem{example}[theorem]{Example}

\theoremstyle{remark}
\newtheorem{remark}[theorem]{Remark}

\numberwithin{equation}{section}

\input xy
\xyoption{all}

\newcommand{\bC}{\mathbb{C}}

\newcommand{\bN}{{\mathbb{N}}}
\newcommand{\bZ}{{\mathbb{Z}}}
\newcommand{\bQ}{{\mathbb{Q}}}
\newcommand{\bR}{\mathbb{R}}

\newcommand{\Ass}{\mbox{\rm{Ass}} }
\newcommand{\Min}{\mbox{\rm{Min}} }

\newcommand{\hlt}{\mathrm {ht\, }}

\newcommand{\dm}{\mathrm {dim }}

\newcommand{\id}{\mathrm {id }}

\newcommand{\Supp}{\mbox{\rm{Supp}} }
\newcommand{\supp}{\mbox{\rm{supp}} }
\newcommand{\spec}{\mathrm {Spec\, }}

\newcommand{\kr}{\mathrm {Ker\, }}

\newcommand{\Hom}{\mbox{\rm{Hom}} }
\newcommand{\Ext}{\mbox{\rm{Ext}} }

\newcommand{\fM}{{\mathfrak m}}

\newcommand{\fp}{\mathfrak{p}}
\newcommand{\fQ}{\mathfrak{q}}

\newcommand{\fa}{\mathfrak{a}}

\newcommand{\cD}{{\mathcal D}}

\newcommand{\cM}{{\mathcal M}}

\newcommand{\la}{{\lambda}}

\newcommand{\lra}{{\longrightarrow}}

\newcommand{\bop}{{\bigoplus}}
\newcommand{\pa}{{\partial}}

\newcommand{\mc}{\mathcal}
\newcommand{\E}[1]{\mbox{E}_{R}(R/#1)}

\begin{document}

\title[Lyubeznik numbers of monomial ideals]{Lyubeznik numbers of monomial ideals}

\author[J. \`Alvarez Montaner]{Josep \`Alvarez Montaner$^\dag$}
\thanks{$^\dag$Partially supported by MTM2010-20279-C02-01 and SGR2009-1284}
\address{Dept. Matem\`atica Aplicada I\\
Universitat Polit\`ecnica de Catalunya\\ Av. Diagonal 647, Barcelona
08028, SPAIN} \email{Josep.Alvarez@upc.es}

\author[A. Vahidi]{Alireza Vahidi}
\address{ 
Dept. Mathematics \\Payame Noor University\\ 19395-4697 Tehran, I.R. of IRAN} \email{vahidi.ar@gmail.com}




\begin{abstract}
Let $R=k[x_1,...,x_n]$ be the polynomial ring in $n$ independent
variables, where $k$ is a field. In this work we will study Bass
numbers of local cohomology modules $H^r_I(R)$ supported on a
squarefree monomial ideal $I\subseteq R$. Among them we are mainly
interested in Lyubeznik numbers. We build a dictionary between the
modules $H^r_I(R)$ and the minimal free resolution of the Alexander
dual ideal $I^{\vee}$ that allow us to interpret Lyubeznik numbers
as the obstruction to the acyclicity of the linear strands of $I^{\vee}$.
The methods we develop also help us to give a bound for the injective dimension
of the local cohomology modules in terms of the dimension of the small support.

\end{abstract}

\maketitle

\section{Introduction}

Some finiteness properties of local cohomology modules $H^r_I(R)$
were established by C.~Huneke and R.~Y.~Sharp \cite{HS93} and
G.~Lyubeznik \cite{Ly93,Ly97} for the case of regular local rings
$(R,\fM, k)$ containing a field. Among these properties they proved
a bound for the injective dimension $${\rm id}_R(H_I^r(R)) \leq
{\dm}_R H_I^r(R)$$ and the finiteness of  all the  Bass numbers
$\mu_p(\fp,H_I^{r}(R)):= {\dm}_{k(\fp)}
{\Ext}_{{R}_{\fp}}^p(k(\fp),H_{I{R}_{\fp} }^{r}({R}_{\fp}))$  with
respect to any prime ideal $\fp \subseteq R$. This last fact
prompted G.~Lyubeznik  to define a new set of numerical invariants
$\la_{p,i}(R/I):=\mu_p(\fM,H_I^{n-i}(R)),$ where $n$ is the
dimension of $R$. These invariants satisfy $\la_{d,d}\neq 0$ and
$\la_{p,i}=0$ for $i>d$, $p>i$, where $d=\dm R/I$. Therefore we can
collect them in the following table:
$$\Lambda(R/I)  = \left(
                    \begin{array}{ccc}
                      \la_{0,0} & \cdots & \la_{0,d}  \\
                       & \ddots & \vdots \\
                       &  & \la_{d,d} \\
                    \end{array}
                  \right)
$$

Lyubeznik numbers carry some interesting topological information
(see \cite{Ly93}, \cite{GS98}, \cite{BlBo}, \cite{Bl3}) but not too
many examples can be found in the literature. We point out that a
general algorithm to compute these invariants in characteristic zero
has been given by U.~Walther \cite{Wa99} using the theory of
$D$-modules, i.e. the theory of modules over the ring of $k$-linear
differential operators $D_{R|k}$.

\vskip 2mm

The $D$-module approach was also used by the first author in
\cite{Al00,Al05} to study local cohomology modules supported on
monomial ideals over the polynomial ring $R=k[x_1,...,x_n]$ and
compute Lyubeznik numbers using the so-called {\it characteristic
cycle}. Local cohomology modules supported on monomial ideals
$H_I^r(R)$ have also been extensively studied using their natural
structure as $\bZ^n$-graded modules. For example, N.~Terai \cite{Te}
gives a formula for its graded pieces equivalent, using local
duality with monomial support (see \cite[\S $6.2$]{Mi00}), to the
famous Hochster formula for the $\bZ^n$-graded Hilbert function of
$H_{\fM}^r(R/I)$ \cite{Sta96}. Simultaneously, M.~Musta\c{t}\u{a}
\cite{Mu00} gives a complete description of the $\bZ^n$-graded
structure, i.e.  a formula for the graded pieces of $H_I^r(R)$ and a
description of the linear maps among them. This description is
equivalent to the one given by H.~G.~Gr\"abe \cite{Gra84} to
describe the module structure of $H_{\fM}^r(R/I)$. In the same
spirit, a formula for the graded pieces of $H_J^r(R/I)$, where
$J\supseteq I$ is another squarefree monomial ideal was given by
V.~Reiner, V.~Welker and K.~Yanagawa in \cite{RWY}.

\vskip 2mm

Building on previous work on {\it squarefree} modules \cite{Ya00},
K.~Yanagawa develops in \cite{Ya01} the theory of {\it straight}
modules to study local cohomology modules $H_I^r(R)$ and their Bass
numbers. Simultaneously, E.~Miller \cite{Mi00} also generalized
squarefree modules by introducing the categories of {\bf
a}-positively determined (resp. {\bf a}-determined)
modules\footnote{Squarefree (resp. straight) modules correspond to
{\bf 1}-positively determined (resp. {\bf 1}-determined) modules.}.
When dealing with Bass numbers, K.~Yanagawa gave the following formula
for Lyubeznik numbers:
$$\la_{p,i}(R/I)=\dm_k[{\Ext}_R^{n-p}({\Ext}_R^{n-i}(R/I,R),R)]_{\bf 0}$$
here $[\cdot ]_{\bf 0}$ denotes the degree $0$ component of a $\bZ^n$-graded module.

\vskip 2mm

The approach we take in this work to study  Lyubeznik numbers uses
the fact that they can be realized as the dimension of the degree ${\bf  1}$ part of
the local cohomology modules $H_{{\fM}}^p(H_I^r(R))$. In Section $3$
we compute these graded pieces and, in general, the graded pieces of
$H_{{\fp}}^p(H_I^r(R))$, where $\fp$ is any homogeneous prime ideal.
More precisely, the piece $[H_{{\fM}}^p(H_I^r(R))]_{\bf 1}$ is
nothing but the $p-th$ homology group of a complex of $k$-vector
spaces we construct using the whole structure of $H_I^r(R)$, i.e.
the graded pieces and the linear maps among them.

\vskip 2mm

In Section $4$ we build a dictionary between local cohomology
modules and free resolutions of monomial ideals that gives us a very
simple interpretation of Lyubeznik numbers. It turns out that the
complex we use to compute the degree ${\bf 1}$ part of
$H_{{\fM}}^p(H_I^r(R))$ is the dual, as $k$-vector spaces, of the complex given
by the scalar entries in the monomial matrices (in the sense of \cite{Mi00,MS05}) of the $r$-linear strand
of the Alexander dual ideal $I^{\vee}$. Thus, Lyubeznik numbers can
be thought as a measure of the acyclicity of these linear strands.

\vskip 2mm

Using the techniques we developed previously we are able to study
some properties of Bass numbers of local cohomology modules in
Section $5$. Recall that, given a finitely generated module $M$, one
has  $ {\rm id}_R \hskip 1mm M \geq {\dim}_R \hskip 1mm \Supp_R M$.
This bound is a consequence of the following well-known property:
Let $\fp \subseteq \fQ \in {\rm Spec} R$ such that $\hlt
(\fQ/\fp)=s$. Then $${\mu}_i(\fp, M)\neq 0 \Longrightarrow
\mu_{i+s}(\fQ, M)\neq 0.$$ For the case of local cohomology modules
this property is no longer true but we can control the behavior of
Bass numbers depending on the structure of $H_I^r(R)$. This control
leads to  a sharper bound for the injective dimension of local
cohomology modules supported on monomial ideals in terms of the
dimension of the small support of these modules $$ {\rm id}_R
H_I^r(R) \leq {\dim}_R \supp_R H_I^r(R).$$ We recall that the small
support was introduced by H.~B.~Foxby \cite{Fox} and consists on the
prime ideals having a Bass number different from zero. For finitely
generated modules the small support coincide with the support but
this is no longer true for non-finitely generated modules.

\vskip 2mm

In Section $6$ we use a shifted version of graded Matlis duality to
study dual Bass numbers. We obtain analogous results to those
obtained for Bass numbers that allow us to study projective
resolutions of local cohomology modules.

\vskip 2mm

\noindent {\it Acknowledgement:} We would like to thank
O.~Fern\'andez-Ramos for implementing our methods using the package
{\tt Macaulay 2} \cite{GS}. We also thank E.~Miller and K.~Yanagawa
for several useful remarks and clarifications and J.~Herzog for
pointing us out to sequentially Cohen-Macaulay ideals as those
ideals having trivial Lyubeznik numbers.

\section{Local cohomology modules supported on monomial ideals}

Let $R=k[x_1,...,x_n]$ be the polynomial ring in $n$ independent
variables, where $k$ is a field. An ideal $I\subseteq R$ is said to
be a { squarefree monomial ideal} if it may be generated by
squarefree monomials  ${\bf x^{\alpha}}:= x_1^{\alpha_1}\cdots
x_n^{\alpha_n},\hskip 2mm {\rm where}\hskip 2mm {\bf \alpha}\in
\{0,1\}^n$. Its minimal primary decomposition is given in terms
of { face ideals} ${\fp_{\alpha}}:= \langle x_i\hskip 2mm | \hskip
2mm \alpha_i \neq 0 \rangle, \hskip 2mm  {\bf \alpha}\in \{0,1\}^n.$
For simplicity we will denote the homogeneous maximal
ideal $\fM:=\fp_{{\bf 1}}=(x_1,\dots,x_n)$, where ${\bf 1}=(1,\dots,1)$. As usual,  we denote $|\alpha|= \alpha_1
+\cdots+\alpha_n$ and $\varepsilon_1,\dots, \varepsilon_n$ will be
the natural basis of $\bZ^n$.

\vskip 2mm

A lot of progress in the study of local cohomology modules
$H_I^r(R)$  supported on monomial ideals has been made based on the
fact that they have a structure as $\bZ^n$-graded modules. Another
line of research uses their structure as regular holonomic modules
over the ring of $k$-linear differential operators $D_{R|k}$, in
particular the fact that they are finitely generated. The aim of
this Section is to give a quick overview of both approaches. For the
$\bZ^n$-graded case we will highlight the main results obtained in
\cite{Mu00}, \cite{Te}, \cite{Ya01} (see also \cite{MS05}). The main
sources for the $D_{R|k}$-module case are \cite{AGZ03}, \cite{AZ}.
For unexplained terminology in the theory of $D_{R|k}$-modules one
may consult \cite{Bj79}, \cite{Co95}.

\subsection{$\bZ^n$-graded structure}
Local cohomology modules $H_I^r(R)$ supported on monomial ideals are
$\bZ^n$-graded modules satisfying some nice properties since they
fit, modulo a shifting by ${\bf 1}$, into the category of straight
(resp. {\bf 1}-determined) modules introduced by K.~Yanagawa
\cite{Ya01} (resp. E.~Miller \cite{Mi00}). In this framework, these
modules are completely described by the graded pieces
$H_I^r(R)_{-\alpha}$ for all $\alpha \in \{0,1\}^n$ and the
morphisms given by the multiplication by $x_i$:$$\cdot x_i:
H_{I}^{r}(R)_{-\alpha} \longrightarrow H_{I}^{r}(R)_{-(\alpha -
\varepsilon_i) }$$

N.~Terai \cite{Te} gave a description of these graded pieces as
follows:
$$H_I^{r}(R)_{-\alpha}\cong
\widetilde{H}_{n-r-|\alpha|-1}({\rm link}_\alpha \Delta; k),$$
\noindent where $\Delta$ is the simplicial complex on the set of
vertices $\{x_1,\dots,x_n\}$ corresponding to the squarefree
monomial ideal $I$ via the Stanley-Reisner correspondence and, given
a face $\sigma_{\alpha}:=\{x_i \hskip 2mm | \hskip 2mm
\alpha_i=1\}\in \Delta$, the link of $\sigma_{\alpha}$ in $\Delta$
is \hskip 2mm $${\rm link}_{{\alpha}}\Delta:=\{\tau \in \Delta
\hskip 2mm | \hskip 2mm \sigma_{\alpha} \cap \tau = \emptyset,
\hskip 2mm \sigma_{\alpha} \cup \tau \in \Delta \}.$$

A different approach was given independently by M.~Musta\c{t}\u{a}
\cite{Mu00} in terms of the restriction to $\sigma_{\alpha}$ that we
denote \hskip 2mm $\Delta_{{\alpha}}:=\{\tau \in \Delta \hskip 2mm |
\hskip 2mm \tau \in \sigma_{\alpha} \}.$ We have:
$$H_I^{r}(R)_{-\alpha}\cong \widetilde{H}^{r-2}(\Delta^{\vee}_{{\bf 1}-\alpha};
k), $$ where $\Delta^{\vee}_{{\bf 1}-\alpha}$ denotes the Alexander
dual of $\Delta_{{\bf 1}-\alpha}$. Both approaches are equivalent
since the equality of simplicial complexes $\Delta^{\vee}_{{\bf
1}-\alpha}=({\rm link}_{{\alpha}}\Delta)^{\vee}$ induces, by
Alexander duality, the isomorphism
$$\widetilde{H}_{n-r-|\alpha|-1}({\rm link}_{{\alpha}}\Delta; k) \cong \widetilde{H}^{r-2}(\Delta^{\vee}_{{\bf 1}-\alpha};
k).$$ Musta\c{t}\u{a} also describes the  multiplication morphism
$\cdot x_i: H_{I}^{r}(R)_{-\alpha} \longrightarrow
H_{I}^{r}(R)_{-(\alpha - \varepsilon_i) }$. It corresponds to the
morphism
$$\widetilde{H}^{r-2}(\Delta^{\vee}_{{\bf 1}-\alpha -
\varepsilon_i}; k)\longrightarrow
\widetilde{H}^{r-2}(\Delta^{\vee}_{{\bf 1}-\alpha}; k),$$ induced by
the inclusion $\Delta^{\vee}_{{\bf 1}-\alpha - \varepsilon_i}
\subseteq \Delta^{\vee}_{{\bf 1}-\alpha}$.

\subsection{$D$-module structure}
Local cohomology modules $H^r_I(R)$ supported on monomial ideals
also satisfy nice properties when viewed as $D_{R|k}$-modules since
they belong to the subcategory $D_{v=0}^T$ of regular holonomic
$D_{R|k}$-modules with support a normal crossing $T:=\{x_1\cdots
x_n=0\}$ and variation zero defined in \cite{AGZ03}. An object $M$
of this category is characterized by the existence of an increasing
filtration $\{ {F_j}\}_{0\le j\le n}$ of submodules of $M$ such that
there are isomorphisms of $D_{R|k}$-modules
$$
{F_j}/{F_{j-1}} \simeq \bop_{\stackrel{}{\scriptscriptstyle{ \mid
\alpha \mid = j}}}\,({H}_{\fp_{\alpha}}^{\mid \alpha
\mid}(R))^{m_{\alpha}},
$$ for some integers $m_{\alpha}\ge 0$,  $\alpha \in \{0,1\}^n$.
We point out that in this category we have the following objects,
$\forall \alpha \in \{0,1\}^n$:

\vskip 2mm

\begin{itemize}

 \item {\bf Simple:}
$
{H}_{\fp_{\alpha}}^{\mid \alpha \mid}(R) \cong \frac{R[
\frac{1}{{\bf x}^\alpha}]}
  { \sum_{\alpha_i =1} R[ \frac{1}{{\bf
x}^{\alpha - \varepsilon_i}}]} \cong
\frac{D_{R|k}}{D_{R|k}(\{x_i\,|\,\alpha_i=1\},
\{\partial_j\,|\,\alpha_j = 0\})}.$

\vskip 2mm

\item {\bf Injective:} $
E_{\alpha}:= {^{\ast}\E{\fp_{\alpha}}({\bf 1})} \cong \frac{R[ \frac{1}{{\bf
x}^{\bf 1}}]}{\sum_{\alpha_i =1} R[ \frac{1}{{\bf x}^{{\bf 1} -
\varepsilon_i}}]} \cong
\frac{{D_{R|k}}}{{D_{R|k}}(\{x_i\,|\,\alpha_i=1\}, \{x_j\partial_j +
1\,|\,\alpha_j = 0\})}.$

\vskip 2mm

\item {\bf Projective:} $
R_{x^{\alpha}} \cong \frac{D_{R|k}}{D_{R|k}(\{x_i \pa_i +1
\,|\,\alpha_i=1\}, \{\partial_j\,|\,\alpha_j = 0\})}.$

\end{itemize}

\vskip 2mm

Following the work of A.~Galligo, M.~Granger and Ph.~Maisonobe
\cite{GGM1, GGM2} one may describe this category as a quiver
representation. More precisely, let $\mathcal{C}^n_{v=0}$ be the
category whose objects are families
$\mathcal{M}:=\{\mathcal{M}_{\alpha}\}_{\alpha\in \{ 0,1\}^n}$ of
finitely dimensional $k$-vector spaces, endowed with linear maps
$$\mathcal{M}_{\alpha} \stackrel{u_{\alpha,i}}{\longrightarrow}
\mathcal{M}_{\alpha + \varepsilon_i} ,$$ for each $\alpha\in\{
0,1\}^n$ such that $\alpha_i=0$. These maps are called canonical
maps, and they are required to satisfy $u_{\alpha,i}\circ
u_{\alpha+\varepsilon_i,j} = u_{\alpha,j}\circ
u_{\alpha+\varepsilon_j,i}$. Such an object will be called an
$n$-hypercube. A morphism between two $n$-hypercubes $\{\mathcal M
_{\alpha}\}_{\alpha}$ and $\{\mathcal N _{\alpha}\}_{\alpha}$ is a
set of linear maps $\{f_{\alpha}: \mathcal{M}_{\alpha}\to
\mathcal{N}_{\alpha}\}_{\alpha}$, commuting with the canonical maps.

\vskip 2mm

There is an equivalence of categories between $D_{v=0}^T$ and
$\mathcal{C}^n_{v=0}$ given by the contravariant exact functor that
sends an object $M$ of $D_{v=0}^T$ to the $n$-hypercube
$\mathcal{M}$ constructed as follows:

\begin{enumerate}
\item [i)]The vertices of the $n$-hypercube
are the $k$-vector spaces $\mathcal{M}_{\alpha}:=\Hom_{D_{R|k}}(M,
E_{\alpha}).$

\item [ii)]The linear maps  $u_{\alpha,i}$ are
induced by the natural epimorphisms $\pi_{\alpha,i}: E_{\alpha} \to E_{\alpha
+ \varepsilon_i}$.

\end{enumerate}

\vskip 2mm

\noindent The irreducibility of $M$ is determined by the extension
classes of the short exact sequences $$0 \lra F_0 \lra F_1 \lra F_1/F_0 \lra 0 $$ \vskip -5mm$$ \vdots $$ $$0 \lra F_{n-1} \lra F_n = M \lra F_n/F_{n-1} \lra 0 $$
\vskip 2mm \noindent associated to the filtration
$\{ {F_j}\}_{0\le j\le n}$ of submodules of $M$. It is shown in
\cite{AGZ03} and \cite{AZ} that these extension classes are uniquely
determined by the linear maps $u_{\alpha,i}$.

\vskip 2mm

It is also worth to point out that if $CC(M)=\sum m_{\alpha} \hskip
1mm T_{X_\alpha}^{\ast} \mathbb{A}_k^n$ is the {\it characteristic
cycle} of $M$, then for all $\alpha\in \{0,1\}^n$ one has the
equality ${\dim}_{k} \mathcal{M}_{\alpha} = m_{\alpha}$ so the
pieces of the $n$-hypercube of a module $M$ are described by the
characteristic cycle of $M$. Finally, the $n$-hypercube
$\{[H_I^r(R)]_{\alpha}\}_{\alpha\in\{0,1\}^n}$ associated to a local
cohomology module $H_I^r(R)$ has been computed in \cite{AZ}.

\subsection{Both approaches are equivalent}
The category $D_{v=0}^T$ of regular holonomic $D_{R|k}$-modules with
variation zero is equivalent to the category of straight modules
shifted by ${\bf 1}$ (see \cite{AGZ03}). Let $ M\in D_{v=0}^T$ and
$\cM\in \mathcal{C}^n_{v=0}$ be the corresponding $n$-hypercube. The
vertices and linear maps of $\cM$ can be described from the graded
pieces of $M$. Let $(M_{-\alpha})^{\ast}$ be the dual of the
$k$-vector space defined by the piece of $M$ of degree $-\alpha$,
$\alpha\in \{0,1\}^n$. Then, there are isomorphisms
$${\mc M}_\alpha \cong (M_{-\alpha})^{\ast}$$
such that the following diagram commutes:
$$\xymatrix{{\mc M}_\alpha\ar[rr]^(.4){u_{\alpha,i}}&&
{\mc M}_{\alpha + \varepsilon_i}\\
{(M_{-\alpha})^{\ast}}\ar[rr]^{(x_i)^{\ast}}\ar[u]^\cong&&
(M_{-\alpha - \varepsilon_i})^{\ast}\ar[u]^\cong}$$ where
$(x_i)^{\ast}$ is the dual of the multiplication by $x_i$.

\vskip 2mm

In this work we are going to use the $D$-module approach just
because of the habit of the first author. In principle this approach
only works for the case of fields of characteristic zero since the
category $\mathcal{C}^n$ described in \cite{GGM1} is defined over
$\bC$ and its subcategory $\mathcal{C}^n_{v=0}$ can be extended to
any field of characteristic zero (see \cite{AZ}). We did not make
any previous mention to the characteristic of the field because the
results are also true in positive characteristic even though we do
not have an analogue to the results of \cite{GGM1, GGM2}. In this
case one has to define modules with variation zero via the
characterization given by the existence of an increasing filtration
$\{ {F_j}\}_{0\le j\le n}$ of submodules of $M$ such that
$$
{F_j}/{F_{j-1}} \simeq \bop_{\stackrel{}{\scriptscriptstyle{ \mid
\alpha \mid = j}}}\,({H}_{\fp_{\alpha}}^{\mid \alpha
\mid}(R))^{m_{\alpha}},
$$ for some integers $m_{\alpha}\ge 0$,  $\alpha \in \{0,1\}^n$.
Finally we point out that, using the same arguments as in
\cite[Lemma 4.4]{AGZ03}, the $n$-hypercube $\mathcal{M}$ associated
to a module with variation zero $M$ should be constructed using the
following variant in terms of graded morphisms

\begin{enumerate}
\item [i)]The vertices of the $n$-hypercube
are the $k$-vector spaces $\mathcal{M}_{\alpha}:=^{\ast}\Hom_{R}(M,
E_{\alpha}).$

\item [ii)]The linear maps  $u_{\alpha,i}$ are
induced by the natural epimorphisms $\pi_{\alpha,i}: E_{\alpha} \to
E_{\alpha + \varepsilon_i}$.

\end{enumerate}

\vskip 2mm

From now on we will loosely use the term {\it pieces of a module
$M$} meaning the pieces of the $n$-hypercube associated to $M$. If
the reader is more comfortable with the $\bZ^n$-graded point of view
one may also reformulate all the results in this paper using the
$\bZ^n$-graded pieces of $M$ (with the appropriate sign). One only
has to be careful with the direction of the arrows in the complexes
of $k$-vector spaces we will construct in the next Sections.

\begin{remark}
The advantage of the $D$-module approach is that it is more likely
to be extended to other situations like the case of hyperplane
arrangements. We recall that local cohomology modules with support
an arrangement of linear subvarieties were already computed in
\cite{AGZ03} and a quiver representation of $D_{R|k}$-modules with
support a hyperplane arrangement is given in \cite{Ko}, \cite{KV}.
\end{remark}

\section{Local cohomology of modules with variation zero}

Let $M\in D_{v=0}^T$ be a regular holonomic $D_{R|k}$-module with
variation zero. The aim of this Section is to compute the pieces of
the local cohomology module $H_{{\fp}_\alpha}^p(M)$, for any given
homogeneous prime ideal $\fp_\alpha$, $\alpha\in \{0, 1\}^{n}$. This
module also belongs to $D_{v=0}^T$ so we want to compute the pieces
of the corresponding $n$-hypercube
$\{[H_{{\fp}_\alpha}^p(M)]_{\beta}\}_{\beta\in \{0,1\}^n} \in
\mathcal{C}^n_{v=0}$. Among these pieces we find the Bass numbers of
$M$ (see \cite{Al05}). Namely, we have
$$\mu_p({\fp}_\alpha,M)= {\rm dim}_k \hskip 1mm
[H_{{\fp}_\alpha}^p(M)]_{\alpha}$$ Bass numbers have a good behavior
with respect to localization so we can always assume that $
\fp_\alpha=\fM$ is the maximal ideal and $\mu_p({\fM},M)= {\rm
dim}_k \hskip 1mm [H_{{\fM}}^p(M)]_{\bf 1}$.

\begin{remark}
Let $\mathcal{M} \in \mathcal{C}^n_{v=0}$ be an $n$-hypercube. The
restriction of $\mathcal{M}$ to a face ideal $\fp_\alpha$,
$\alpha\in \{0, 1\}^{n}$ is the $|\alpha|$-hypercube
$\mathcal{M_{\leq \alpha}}:=\{\mathcal{M}_{\beta}\}_{\beta\leq
\alpha} \in \mathcal{C}^{|\alpha|}_{v=0}$ (see \cite{Al05}). This
gives a functor that in some cases plays the role of the
localization functor. In particular, to compute the Bass numbers
with respect to $\fp_\alpha$ of a module with variation zero $M$ we
only have to consider the corresponding $|\alpha|$-hypercube
$\mathcal{M_{\leq \alpha}}$ so we may assume that $ \fp_\alpha$ is
the maximal ideal.
\end{remark}

In Section $4$ we will specialize to the case of $M$ being a local
cohomology module $H^r_I(R)$.

\subsection{The degree {\bf $1$} piece of $H_{{\fM}}^p(M)$}
We start with his particular case since it is more enlightening than
the general one. Using the whole structure of $M$. i.e. the pieces of $M$
and the linear maps between them, we want to construct a complex of $k$-vector
spaces  whose homology is $[H_{{\fM}}^p(M)]_{\bf 1}$.

\vskip 2mm

The degree ${\bf 1}$ part of the hypercube corresponding to the
local cohomology module $H_{\fM}^p(M)$ is the $p$-th homology of the
complex of $k$-vector spaces $$[{{\check{C}}}_{\fM}(M)]_{\bf
1}^\bullet: 0\longleftarrow [M]_{\bf
1}\stackrel{\overline{d_{0}}}\longleftarrow \bigoplus_{|\alpha|=1}
[M_{{\bf x^{\alpha}}}]_{\bf
1}\stackrel{\overline{d_{1}}}\longleftarrow
\cdots\stackrel{\overline{d_{p-1}}}\longleftarrow
\bigoplus_{|\alpha|=p} [M_{{\bf x^{\alpha}}}]_{\bf
1}\stackrel{\overline{d_{p}}}\longleftarrow \cdots
\stackrel{\overline{d_{n-1}}}\longleftarrow [M_{{\bf x}^{\bf
1}}]_{\bf 1}\longleftarrow 0$$ that we obtain applying the exact
functor $\Hom_{D_{R|k}}(\cdot , E_{\bf 1})$ to the \v{C}ech complex
$${{\check{C}}}_{\fM}^\bullet(M): 0\longrightarrow
M\stackrel{d_{0}}\longrightarrow \bigoplus_{|\alpha|=1} M_{{\bf
x^{\alpha}}}\stackrel{d_{1}}\longrightarrow
\cdots\stackrel{d_{p-1}}\longrightarrow \bigoplus_{|\alpha|=p}
M_{{\bf x^{\alpha}}}\stackrel{d_{p}}\longrightarrow \cdots
\stackrel{d_{n-1}}\longrightarrow M_{{\bf x}^{\bf 1}}\longrightarrow
0,$$ where the map between summands $M_{{\bf x^{\alpha}}} \lra
M_{{{\bf x}^{\alpha+\varepsilon_i}}}$ is
sign$(i,\alpha+\varepsilon_i)$ times the canonical localization
map\footnote{sign$(i,\alpha)= (-1)^{r-1}$ if $\alpha_i$ is the
$r^{{\rm th}}$ component of $\alpha$ different from zero}.
 On the other hand, giving the appropriate sign to the canonical maps of the hypercube  $\cM=\{[M]_{\alpha}\}_\alpha$ associated to $M$ we can construct the following
complex of $k$-vector spaces:
$${\cM}^\bullet: 0\longleftarrow [M]_{\bf 1}\stackrel{u_{0}}\longleftarrow
\bigoplus_{|\alpha|=n-1} [M]_{\alpha}\stackrel{u_{1}}\longleftarrow
\cdots\stackrel{u_{p-1}}\longleftarrow \bigoplus_{|\alpha|=n-p}
[M]_{\alpha}\stackrel{u_{p}} \longleftarrow \cdots
\stackrel{u_{n-1}}\longleftarrow [M]_{\bf 0}\longleftarrow 0$$ where the
map between summands $[M]_{\alpha} \lra [M]_{\alpha+\varepsilon_i}$ is
sign$(i,\alpha+\varepsilon_i)$ times the canonical map $u_{\alpha,i}$.

\begin{example} $3$-hypercube and its associated complex
{\tiny $${\xymatrix { &M_{(0,0,0)} \ar[dl]_{u_1} \ar[d]_{u_2} \ar[dr]^{u_3} &
\\ M_{(1,0,0)} \ar[d]_{u_2} \ar[dr]^(.3){u_3}  & M_{(0,1,0)}\ar[dr]_(.3){u_3}|\hole \ar[dl]^(.3){u_1}|\hole & M_{(0,0,1)}  \ar[dl]_(.3){u_1} \ar[d]^{u_2}
\\ M_{(1,1,0)} \ar[dr]_{u_3}& M_{(1,0,1)} \ar[d]_{u_2} & M_{(0,1,1)} \ar[dl]^{u_1}
\\& M_{(1,1,1)} & }} $$ }

{ $$ {\xymatrix{0& M_{(1,1,1)} \ar[l] && {\begin{array}{l}
 M_{(1,1,0)}\\
\hskip 4mm \oplus \\
M_{(1,0,1)} \\
\hskip 4mm \oplus \\
M_{(0,1,1)}
 \end{array}}
   \ar[ll]_(.5){(u_3, u_2, u_1)}&&&
{\begin{array}{l}
 M_{(1,0,0)}\\
\hskip 4mm \oplus \\
M_{(0,1,0)} \\
\hskip 4mm \oplus \\
M_{(0,0,1)}
 \end{array}}
 \ar[lll]_{{\tiny \begin{pmatrix}
 -u_2 & -u_1 & 0   \\
 u_3 & 0 & -u_1     \\
 0 & u_3 & u_2
\end{pmatrix}}}&&
M_{(0,0,0)} \ar[ll]_{{\tiny \begin{pmatrix}
  -u_1   \\
 u_2    \\
-u_3
\end{pmatrix}}} & 0  \ar[l]     }}   $$   }

\end{example}

The main result of this Section is  the following

\begin{proposition}
Let $M\in D_{v=0}^T$ be a regular holonomic $D_{R|k}$-module with
variation zero and $\cM^{\bullet}$ its corresponding complex associated to the $n$-hypercube. Then, there is an isomorphism of complexes ${\cM}^\bullet\cong
[{\check{C}}_{\fM}(M)]_{\bf 1}^\bullet.$ In particular $[H_{{\fM}}^p(M)]_{\bf 1}\cong {\rm
H}_{p}(\cM^{\bullet})$.
\end{proposition}

Therefore we have the following characterization of Bass numbers:

\begin{corollary}
Let $M\in D_{v=0}^T$ be a regular holonomic $D_{R|k}$-module with
variation zero and $\cM^{\bullet}$ its corresponding complex
associated to the $n$-hypercube. Then $$\mu_p({\fM},M)= {\rm dim}_k
\hskip 1mm {\rm H}_{p}(\cM^{\bullet})$$
\end{corollary}

\begin{proof}
Using \cite[Prop. 3.2]{Al05} one may check out that the $k$-vector
spaces $[M_{{\bf x}^{\alpha}}]_{\bf 1}$ and $ [M]_{{\bf 1}-\alpha}$
have the same dimension. An explicit isomorphism
$\phi_\alpha:[M]_{{\bf 1}-\alpha} \lra [M_{{\bf x}^{\alpha}}]_{\bf
1}$ is defined as follows:

\vskip 2mm

Let $f\in [M]_{{\bf 1}-\alpha}= \Hom_{D_{R|k}}(M , E_{{\bf 1}-\alpha})$, then
$\phi_\alpha(f)\in [M_{{\bf x}^{\alpha}}]_{\bf 1}=
\Hom_{D_{R|k}}(M_{{\bf x}^{\alpha}}, E_{{\bf 1}})$ is the
composition
$$M_{{\bf x}^\alpha}\stackrel{f_{{\bf x}^\alpha}}\longrightarrow
(E_{{\bf 1}-\alpha})_{{\bf
x}^\alpha}\stackrel{\theta_{\alpha}^{-1}}\longrightarrow E_{{\bf
1}-\alpha}\stackrel{\pi_\alpha} \longrightarrow E_{{\bf 1}}$$

\noindent where:

\vskip 2mm

\begin{itemize}

 \item [$\cdot$] $ f_{{\bf x}^\alpha}: M_{{\bf x}^\alpha}\longrightarrow
(E_{{\bf 1}-\alpha})_{{\bf
x}^\alpha}$ is the localization of $f$.

\vskip 2mm

\item [$\cdot$] $\theta_\alpha:  E_{{\bf 1}-\alpha} \longrightarrow (E_{{\bf 1}-\alpha})_{{\bf x}^\alpha}$ is the natural
localization map.

\vskip 2mm

\item [$\cdot$] $\pi_\alpha: E_{{\bf 1}-\alpha} \longrightarrow E_{{\bf 1}} $  is the natural epimorphism.

\end{itemize}

\vskip 2mm

{\bf Claim: }  $\theta_{\alpha}$ is an isomorphism.

\vskip 2mm

\begin{proof}[Proof of Claim:]
When $\alpha={\bf 1}$ we have $E_{\bf 0} \cong R_{{\bf x}^{1}}$ so
the result follows. For $\alpha\neq {\bf 1}$, let $t \in \bZ_{\geq
0}$  and $m= \overline{\frac{\sum_{\beta\in \mathbb{Z}^{n}}a_\beta
{\bf x}^\beta}{{\bf x}^{{\bf 1}\cdot t}}}$ be an element of $E_{{\bf
1}-\alpha}$ such that $\theta_{\alpha}(m)= 0$. There exists $s\in
\bZ_{\geq 0}$ such that $$0={\bf x}^{\alpha \cdot s}m =
\overline{\frac{\sum_{\beta\in \mathbb{Z}^{n}}a_\beta {\bf x}^{\beta
+ \alpha\cdot s}}{{\bf x}^{{\bf 1}\cdot t}}}$$ so, there exists $i$ such that
$\alpha_i=0$ and $\beta_i + {\alpha_i}\cdot s \geq  t$. Thus
$\beta_i  \geq  t$ and $m=0$ so $\theta_{\alpha}$ is a monomorphism. Now, let $m'= \overline{\frac{\frac{\sum_{\beta\in
\mathbb{Z}^{n}}a_\beta {\bf x}^\beta}{{\bf x}^{{\bf 1}\cdot
t}}}{{\bf x}^{\alpha \cdot s}}}$ be an element of $(E_{{\bf
1}-\alpha})_{{\bf x}^\alpha}$. Then $m'=\theta_{\alpha}(m)$, where
$m= \overline{\frac{\sum_{\beta\in \mathbb{Z}^{n}}a_\beta {\bf
x}^{\beta +({\bf 1}-\alpha)\cdot s}}{{\bf x}^{{\bf 1}\cdot (t+s)}}}$
\end{proof}

Now we check out that $\phi_\alpha$ is an isomorphism. Recall that
$[M]_{{\bf 1}-\alpha}$ and $[M_{{\bf x}^{\alpha}}]_{\bf 1}$ have
same dimension so it is enough to prove that $\phi_\alpha$ is a
monomorphism. Consider $f\in [M]_{{\bf 1}-\alpha}$ such that
$\phi_\alpha(f)=0$. There exists $t\in \bZ_{\geq 0}$ such that
$f(m)= \overline{\frac{\sum_{\beta\in \mathbb{Z}^{n}}a_\beta {\bf
x}^\beta}{{\bf x}^{{\bf 1}\cdot t}}} \in E_{{\bf 1}-\alpha}$, for a
given $m\in M$. Then:
$$0= \phi_{\alpha}(f)(\frac{m}{{\bf x}^{{\alpha}\cdot s}})
= \pi_{\alpha}\theta^{-1}_{\alpha}(\frac{f(m)}{{\bf
x}^{{\alpha}\cdot s}}) =
 \pi_{\alpha}\theta^{-1}_{\alpha}(\frac{\overline{\frac{\sum_{\beta\in \mathbb{Z}^{n}}a_\beta {\bf
x}^\beta}{{\bf x}^{{\bf 1}\cdot t}}}}{{\bf x}^{{\alpha}\cdot s}}) =
\pi_{\alpha}(\overline{\frac{\sum_{\beta\in \mathbb{Z}^{n}}a_\beta
{\bf x}^{\beta +({{\bf 1}-\alpha})s}}{{\bf x}^{{\bf 1}\cdot
(t+s)}}})= $$ \hskip 3.6cm $=\overline{\frac{\sum_{\beta\in
\mathbb{Z}^{n}}a_\beta {\bf x}^{\beta +({{\bf 1}-\alpha})s}}{{\bf
x}^{{\bf 1}\cdot (t+s)}}}$

\vskip 2mm

\noindent Thus, there exists $1\leq i \leq n$ such that $\beta_i + ({\bf
1}-\alpha)_i s \geq t+s$. If we take $s$ big enough (e.g. $s>{\rm max}\{|\beta_i -t|,
i=1,\dots,n\}$), it follows that $({\bf 1}-\alpha)_i= 1$ and $\beta_i \geq t$. Hence $f(m)=
\overline{\frac{\sum_{\beta\in \mathbb{Z}^{n}}a_\beta {\bf
x}^\beta}{{\bf x}^{{\bf 1}\cdot t}}}=0$ so $f=0$ as desired.

\vskip 2mm

 To finish the proof we have to check out that the diagram
$$\xymatrix{\bigoplus_{|\alpha|=n-p}[M]_{\alpha} \ar[d]^{\oplus \phi_{\alpha}}&&&
\bigoplus_{|\alpha|=n-(p+1)}[M]_{\alpha} \ar[lll]_(.5){u_{p}} \ar[d]^{\oplus \phi_{\alpha}}\\
\bigoplus_{|\alpha|=p} [M_{{\bf x}^{\alpha}}]_{{\bf 1}}&&&
\bigoplus_{|\alpha|=p+1}[M_{{\bf x}^{\alpha}}]_{{\bf
1}}\ar[lll]^{\overline{d_p}}}$$ is commutative. Restricting to the corresponding summands it is enough to consider the following diagram $$\xymatrix{ \Hom_{D_{R|k}}(M, E_{{\bf 1}-\alpha}) \ar[d]^{\phi_{\alpha}}&&&
\Hom_{D_{R|k}}(M, E_{{\bf 1}-(\alpha +
              \varepsilon_i)}) \ar[lll]_(.5){(-1)^s u_{{\bf 1}-(\alpha +\varepsilon_i),i}} \ar[d]^{\phi_{\alpha +\varepsilon_i}}\\
\Hom_{D_{R|k}}(M_{{\bf x}^{\alpha}}, E_{{\bf 1}})&&&
\Hom_{D_{R|k}}(M_{{\bf x}^{\alpha +\varepsilon_i}}, E_{{\bf
1}})\ar[lll]^{(-1)^s \overline{\vartheta_{\alpha,i}}}}$$ where $\vartheta_{\alpha,i}:M_{{\bf x}^{\alpha}} \lra M_{{\bf x}^{\alpha +\varepsilon_i}}$ is the natural localization map.

\vskip 2mm

For $f\in \Hom_{D_{R|k}}(M, E_{{\bf 1}-(\alpha + \varepsilon_i)})$ the morphisms
$\phi_\alpha(u_{{\bf 1}-(\alpha +\varepsilon_i),i}(f))$ and $\overline{\vartheta_{\alpha,i}}(\phi_{\alpha +\varepsilon_i}(f)) $ are, respectively, the
compositions
$$M_{{\bf x}^\alpha}\stackrel{f_{{\bf x}^\alpha}}\longrightarrow
(E_{{\bf 1}-(\alpha + \varepsilon_i)})_{{\bf
x}^\alpha}\stackrel{{(\pi_i)}_{{\bf x}^\alpha}}\longrightarrow
(E_{{\bf 1}-\alpha})_{{\bf
x}^\alpha}\stackrel{\theta_{\alpha}^{-1}}\longrightarrow E_{{\bf
1}-\alpha}\stackrel{\pi_\alpha} \longrightarrow E_{{\bf 1}}$$

$$M_{{\bf x}^\alpha}\stackrel{\vartheta_{\alpha,i}}\longrightarrow
M_{{\bf x}^{\alpha + \varepsilon_i}}\stackrel{f_{{\bf x}^{\alpha + \varepsilon_i}}}\longrightarrow
(E_{{\bf 1}-{(\alpha + \varepsilon_i)}})_{{\bf
x}^\alpha}\stackrel{\theta_{\alpha + \varepsilon_i}^{-1}}\longrightarrow E_{{\bf
1}-{(\alpha +\varepsilon_i)}}\stackrel{\pi_{\alpha + \varepsilon_i}} \longrightarrow E_{{\bf 1}}$$

\vskip 2mm

Let $m\in M$ and $f(m)= \overline{\frac{\sum_{\beta\in \mathbb{Z}^{n}}a_\beta {\bf
x}^\beta}{{\bf x}^{{\bf 1}\cdot t}}} \in E_{{\bf 1}-(\alpha + \varepsilon_i)}$, where $t\in \bZ_{\geq 0}$.
Then, for $s\in \bZ_{\geq 0}$

 $$\phi_\alpha(u_{{\bf 1}-(\alpha +\varepsilon_i),i}(f))(\frac{m}{{\bf x}^{{\alpha}\cdot s}})=
\pi_\alpha(\theta_{\alpha}^{-1}({(\pi_i)}_{{\bf x}^\alpha}(f_{{\bf
x}^\alpha})))(\frac{m}{{\bf x}^{{\alpha}\cdot s}})=
\pi_\alpha(\theta_{\alpha}^{-1}({(\pi_i)}_{{\bf
x}^\alpha}))(\frac{f(m)}{{\bf x}^{{\alpha}\cdot s}})= $$ $$=
\pi_\alpha(\theta_{\alpha}^{-1})(\frac{\overline{\frac{\sum_{\beta\in
\mathbb{Z}^{n}}a_\beta {\bf x}^\beta}{{\bf x}^{{\bf 1}\cdot
t}}}}{{\bf x}^{{\alpha}\cdot s}})
=\pi_{\alpha}(\overline{\frac{\sum_{\beta\in \mathbb{Z}^{n}}a_\beta
{\bf x}^{\beta +({{\bf 1}-\alpha})s}}{{\bf x}^{{\bf 1}\cdot
(t+s)}}})= \overline{\frac{\sum_{\beta\in \mathbb{Z}^{n}}a_\beta
{\bf x}^{\beta +({{\bf 1}-\alpha})s}}{{\bf x}^{{\bf 1}\cdot
(t+s)}}}$$

\vskip 2mm

 on the other hand

$$\overline{\vartheta_{\alpha,i}}(\phi_{\alpha +\varepsilon_i}(f))(\frac{m}{{\bf x}^{{\alpha}\cdot s}})=
\pi_{\alpha + \varepsilon_i}(\theta_{\alpha
+\varepsilon_i}^{-1}(f_{{\bf x}^{\alpha +
\varepsilon_i}}(\vartheta_{\alpha,i})))(\frac{m}{{\bf
x}^{{\alpha}\cdot s}})= \pi_{\alpha + \varepsilon_i}(\theta_{\alpha
+\varepsilon_i}^{-1}(f_{{\bf x}^{\alpha +
\varepsilon_i}}))(\frac{x_i^{s} m}{{\bf x}^{{\alpha}\cdot s}})=
$$ $$ \hskip 2.1cm =\pi_{\alpha + \varepsilon_i}(\theta_{\alpha
+\varepsilon_i}^{-1})(\frac{x_i^{s} f(m)}{{\bf x}^{{\alpha}\cdot
s}})= \pi_{\alpha + \varepsilon_i}(\theta_{\alpha
+\varepsilon_i}^{-1})(\frac{x_i^{s} \overline{\frac{\sum_{\beta\in
\mathbb{Z}^{n}}a_\beta {\bf x}^\beta}{{\bf x}^{{\bf 1}\cdot
t}}}}{{\bf x}^{{\alpha}\cdot s}}) =$$ $$\hskip 2.4cm =\pi_{\alpha +
\varepsilon_i}(\overline{\frac{\sum_{\beta\in \mathbb{Z}^{n}}a_\beta
{\bf x}^{\beta + \varepsilon_i \cdot s +({{\bf 1}-(\alpha +
\varepsilon_i}))s}}{{\bf x}^{{\bf 1}\cdot (t+s)}}})=
\overline{\frac{\sum_{\beta\in \mathbb{Z}^{n}}a_\beta {\bf x}^{\beta
+({{\bf 1}-\alpha})s}}{{\bf x}^{{\bf 1}\cdot (t+s)}}}$$

\vskip 2mm

Thus $\phi_\alpha(u_{{\bf 1}-(\alpha +\varepsilon_i),i}(f)) = \overline{\vartheta_{\alpha,i}}(\phi_{\alpha +\varepsilon_i}(f)) $

\end{proof}

\subsection{The pieces of $H_{{\fp}_\alpha}^p(M)$}
In general, for any given $\alpha, \beta \in \{0,1\}^n$, the degree
$\beta$ part of the hypercube corresponding to $H_{\fp_\alpha}^p(M)$
is the $p$-th homology of the complex of $k$-vector spaces
$[{{\check{C}}}_{\fp_\alpha}(M)]_{\beta}^\bullet$ that we obtain
applying the exact functor $\Hom_{D_{R|k}}(\cdot , E_{\beta})$ to
the \v{C}ech complex ${{\check{C}}}_{\fp_\alpha}^\bullet(M)$
associated to the face ideal $\fp_\alpha$. On the other hand, we can
also associate to the $n$-hypercube of $M$ the complex of $k$-vector
spaces:
$${\cM}_{\alpha,\beta}^\bullet: 0\longleftarrow [M]_{\beta}\stackrel{u_{0}}\longleftarrow
\bigoplus_{\tiny
\begin{tabular}{l}
 $|\gamma|=1$\\
$\gamma \leq \alpha$
           \end{tabular}} [M]_{\beta \backslash \gamma}\stackrel{u_{1}}\longleftarrow
\cdots\stackrel{u_{p-1}}\longleftarrow \bigoplus_{\tiny
\begin{tabular}{l}
 $|\gamma|=p$\\
$\gamma \leq \alpha$
           \end{tabular}}
[M]_{\beta \backslash \gamma}\stackrel{u_{p}} \longleftarrow \cdots
\stackrel{u_{|\alpha|-1}}\longleftarrow [M]_{\beta \backslash \alpha}\longleftarrow 0$$ where $\beta \backslash \alpha \in \{0,1\}^n$  is the vector with components
$(\beta \backslash \alpha)_i := \beta_i$ if $\alpha_i=0$ and $0$ otherwise.
The maps between summands are defined by the corresponding canonical maps.

\vskip 2mm

A description of the pieces of $H_{{\fp}_\alpha}^p(M)$ can be
obtained using the same arguments as in the previous subsection so
we will skip the details. The proofs are a little bit more involved
just because of the extra notation.

\begin{proposition}
Let $M\in D_{v=0}^T$ be a regular holonomic $D_{R|k}$-module with
variation zero and, $\forall \alpha, \beta \in \{0,1\}^n$, ${\cM}_{\alpha,\beta}^\bullet$ its
corresponding complex associated to the $n$-hypercube.
Then,  ${\cM}_{\alpha,\beta}^\bullet \cong [{{\check{C}}}_{\fp_\alpha}(M)]_{\beta}^\bullet$.
In particular $[H_{{\fp_\alpha}}^p(M)]_{\beta}\cong {\rm
H}_{p}({\cM}_{\alpha,\beta}^\bullet)$.
\end{proposition}

\begin{corollary}
Let $M\in D_{v=0}^T$ be a regular holonomic $D_{R|k}$-module with
variation zero and ${\cM}_{\alpha,\alpha}^\bullet$ its corresponding complex associated to the $n$-hypercube. Then $$\mu_p({{\fp_\alpha}},M)= {\rm dim}_k \hskip 1mm {\rm
H}_{p}({\cM}_{\alpha,\alpha}^\bullet)$$
\end{corollary}

\section{Lyubeznik numbers of monomial ideals}

Let $(R,\fM, k)$ be a regular local ring of dimension $n$ containing
a field $k$ and $A$ a local ring which admits a surjective ring
homomorphism $\pi : R \lra A$. G. Lyubeznik \cite{Ly93} defines a
new set of numerical invariants of $A$ by means of the Bass numbers
$\la_{p,i}(A):=\mu_p(\fM,H_I^{n-i}(R))$, where $I= \kr \pi$.  This
invariant depends only on $A$, $i$ and $p$, but neither on $R$ nor
on $\pi$. Completion does not change $\la_{p,i}(A)$ so one can
assume $R=k[[x_1, \dots ,x_n]]$. These invariants satisfy
$\la_{d,d}(A)\neq 0$ and $\la_{p,i}(A)=0$ for $i>d$, $p>i$, where
$d=\dm A$. Therefore we can collect them in what we refer as {\it
Lyubeznik table}:
$$\Lambda(R/I)  = \left(
                    \begin{array}{ccc}
                      \la_{0,0} & \cdots & \la_{0,d}  \\
                       & \ddots & \vdots \\
                       &  & \la_{d,d} \\
                    \end{array}
                  \right)
$$

\vskip 2mm

It is worth to point out that for the case of monomial ideals one
may always assume that $R=k[x_1, \dots ,x_n]$. Then, let
$\cM=\{[H^r_I(R)]_{\alpha}\}_{\alpha\in\{0,1\}^n} $ be the
$n$-hypercube of a local cohomology module $H^r_I(R)$ supported on a
monomial ideal $I\subseteq R$. In this case we have a topological
description of the pieces and linear maps of the $n$-hypercube, e.g.
using M.~Musta\c{t}\u{a}'s approach \cite{Mu00}, the complex of
$k$-vector spaces associated to $\cM$ is:
$${\cM}^\bullet: 0\longleftarrow \widetilde{H}^{r-2}(\Delta^{\vee}_{{\bf 0}};
k)\stackrel{u_{0}}\longleftarrow
\cdots\stackrel{u_{p-1}}\longleftarrow \bigoplus_{|\alpha|=p}
\widetilde{H}^{r-2}(\Delta^{\vee}_{\alpha};
k)\stackrel{u_{p}} \longleftarrow \cdots
\stackrel{u_{n-1}}\longleftarrow \widetilde{H}^{r-2}(\Delta^{\vee}_{{\bf 1}};
k)\longleftarrow 0$$ where the
map between summands $\widetilde{H}^{r-2}(\Delta^{\vee}_{\alpha +
\varepsilon_i}; k)\longrightarrow
\widetilde{H}^{r-2}(\Delta^{\vee}_{\alpha}; k),$ is induced by
the inclusion $\Delta^{\vee}_{\alpha}
\subseteq \Delta^{\vee}_{\alpha + \varepsilon_i}$. In particular, the Lyubeznik numbers
of $R/I$ are $$\la_{p,n-r}(R/I)={\rm dim}_k \hskip 1mm {\rm
H}_{p}({\cM}^\bullet)$$

\vskip 2mm

At this point one may wonder whether there is a simplicial complex,
a regular cell complex, or a CW-complex that supports $\cM^\bullet$
so one may get a Hochster-like formula not only for the pieces of
the local cohomology modules $H^r_I(R)$ but for its Bass numbers as
well. Unfortunately this is not the case in general. To check this
out we will make a detour through the theory of free resolutions of
monomial ideals and we refer to the work of M.~Velasco \cite{Vel08}
to find examples of free resolutions that are not supported by
CW-complexes.

\subsection{Building a dictionary}

The minimal graded free resolution of a monomial ideal $J$ is an
exact sequence of free $\bZ^n$-graded $R$-modules:
$$\mathbb{L}_{\bullet}(J): \hskip 3mm \xymatrix{ 0 \ar[r]& L_{m}
\ar[r]^{d_{m}}& \cdots \ar[r]& L_1 \ar[r]^{d_1}& L_{0} \ar[r]& J
\ar[r]& 0}$$  where the $j$-th term is of the form $$L_j =
\bigoplus_{\alpha \in {\mathbb{Z}^n}}
R(-\alpha)^{\beta_{j,\alpha}(J)},$$ and the matrices of the
morphisms $d_j: L_j\longrightarrow L_{j-1}$ do not contain
invertible elements.  The $\bZ^n$-graded {\it Betti numbers} of $J$
are the invariants $\beta_{j,\alpha}(J)$. Given an integer $r$, the
{\it $r$-linear strand} of $\mathbb{L}_{\bullet}(J)$ is the complex:
$$\mathbb{L}_{\bullet}^{<r>}(J): \hskip 3mm \xymatrix{ 0 \ar[r]&
L_{n-r}^{<r>} \ar[r]^{d_{n-r}^{<r>}}& \cdots \ar[r]& L_1^{<r>}
\ar[r]^{d_1^{<r>}}& L_{0}^{<r>} \ar[r]& 0},$$ where $$L_j^{<r>} =
\bigoplus_{|\alpha|=j+r} R(-\alpha)^{\beta_{j,\alpha}(J)},$$ and the
differentials $d_j^{<r>}: L_j^{<r>}\longrightarrow L_{j-1}^{<r>}$
are the corresponding components of $d_j$. A combinatorial
description of the first linear strand was given in \cite{RW}.

\vskip 2mm

E.~Miller \cite{Mi00, MS05} developed the notion of {\it monomial
matrices} to encode the structure of free, injective and flat
resolutions. These are matrices with scalar entries that keep track
of the degrees of the generators of the summands in the source and
the target. The goal of this Section is to show that the
$n$-hypercube of a local cohomology module $H^r_I(R)$ has the same
information as the $r$-linear strand of the Alexander dual ideal of
$I$. More precisely, we will see that the matrices in the complex of
$k$-vector spaces associated to the $n$-hypercube of $H^r_I(R)$ are
the transpose of the monomial matrices of the $r$-linear
strand\footnote{In the language of \cite{PV11} we would say that the
$n$-hypercube has the same information as the frame of the
$r$-linear strand}.

\vskip 2mm

M.~Musta\c{t}\u{a} \cite{Mu00} already proved the following relation
between the pieces of the local cohomology modules and the Betti
numbers of the Alexander dual ideal $$\beta_{j,\alpha}(I^{\vee})=
\dim _k [H^{|\alpha | -j }_I(R)]_{\alpha}$$ so the pieces of
$H^r_I(R)$ for a fixed $r$ describe the modules and the Betti
numbers of the $r$-linear strand of $I^{\vee}$. To prove the
following proposition one has to put together some results scattered
in the work of K.~Yanagawa \cite{Ya00, Ya01}.

\begin{proposition}
Let $\cM=\{[H^r_I(R)]_{\alpha}\}_{\alpha\in \{0,1\}^n} $ be the
$n$-hypercube of a fixed local cohomology module $H^r_I(R)$
supported on a monomial ideal $I\subseteq R=k[x_1, \dots ,x_n]$.
Then, $\cM^\bullet$ is the complex of $k$-vector spaces whose matrices
are the transpose of the monomial matrices of the $r$-linear strand
$\mathbb{L}_{\bullet}^{<r>}(I^{\vee})$ of
the Alexander dual ideal of $I$.
\end{proposition}

In \cite{Ya00} K.~Yanagawa develops the notion of {\it squarefree
module}, this is a $\bN^n$-graded module $M$ described by the graded
pieces $M_\alpha$, $\alpha\in \{0,1\}^n$ and the morphisms given by
the multiplication by $x_i$. To such a module $M$ he constructs a
chain complex $\mathbb{F}_{\bullet}(M)$ of free $R$-modules as
follows:

$$\mathbb{F}_{\bullet}(M): 0\longrightarrow [M]_{\bf 1}\otimes_k R \stackrel{d_{0}}\longrightarrow
\cdots\stackrel{d_{p-1}}\longrightarrow \bigoplus_{|\alpha|=n-p}
[M]_{\alpha}\otimes_k R \stackrel{d_{p}} \longrightarrow \cdots
\stackrel{d_{n-1}}\longrightarrow [M]_{\bf 0}\otimes_k R
\longrightarrow 0$$ where the map between summands $[M]_{\alpha
+\varepsilon_i}\otimes_k R \lra [M]_{\alpha}\otimes_k R$  sends
$y\otimes 1 \in [M]_{\alpha +\varepsilon_i}\otimes_k R$ to
sign$(i,\alpha+\varepsilon_i) \hskip 1mm  (x_i y\otimes x_i)$. For
the particular case of $M={\rm Ext}_R^{r}(R/I,R(-{\bf 1}))$ he
proved an isomorphism (after an appropiate shifting) between
$\mathbb{F}_{\bullet}(M)$ and the $r$-linear strand
$\mathbb{L}_{\bullet}^{<r>}(I^{\vee})$ of the Alexander dual ideal
$I^{\vee}$ of $I$.

\vskip 2mm

In \cite{Ya01} he proves that the categories of squarefree modules
and straight modules are equivalent. Therefore one may also
construct the chain complex $\mathbb{F}_{\bullet}(M)$ for any
straight module $M$. The squarefree module ${\rm
Ext}_R^{r}(R/I,R(-{\bf 1}))$ corresponds\footnote{In the terminology
of E.~Miller \cite{Mi00} one states that the ${\rm \check{C}}$ech
hull of ${\rm Ext}_R^{r}(R/I,R(-{\bf 1}))$ is $H^r_I(R)(-{\bf 1})$}
to the local cohomology modules $H^r_I(R)(-{\bf 1})$ so there is an
isomorphism between $\mathbb{F}_{\bullet}(H^r_I(R)(-{\bf 1}))$ and
the $r$-linear strand $\mathbb{L}_{\bullet}^{<r>}(I^{\vee})$ after
an appropriate shifting. Taking a close look to the construction of
$\mathbb{F}_{\bullet}(M)$ one may check that the scalar entries in
the corresponding monomial matrices are obtained by transposing the
scalar entries in the one associated to the hypercube of $H^r_I(R)$
with the appropriate shift. More precisely, if
$$\mathbb{L}_{\bullet}^{<r>}(I^{\vee}): \hskip 3mm \xymatrix{ 0
\ar[r]& L_{n-r}^{<r>} \ar[r]& \cdots \ar[r]& L_1^{<r>} \ar[r]&
L_{0}^{<r>} \ar[r]& 0},$$ is the $r$-linear strand of the Alexander
dual ideal $I^\vee$ then we transpose its monomial matrices to
obtain a complex of $k$-vector spaces indexed as follows:
$$\mathbb{F}_{\bullet}^{<r>}(I^{\vee})^{\ast}: \hskip 3mm \xymatrix{ 0 &
K_{0}^{<r>} \ar[l]& \cdots \ar[l]& K_{n-r-1}^{<r>} \ar[l]&
K_{n-r}^{<r>} \ar[l]& 0 \ar[l]}$$

\vskip 2mm

\begin{corollary}
Let $\mathbb{F}_{\bullet}^{<r>}(I^{\vee})^{\ast}$ be the complex of $k$-vector spaces
obtained from the $r$-linear strand of the minimal free resolution of the
Alexander dual ideal $I^{\vee}$ transposing its monomial matrices. Then $$\lambda_{p,n-r}(R/I)= {\rm
dim}_k H_{p}(\mathbb{F}_{\bullet}^{<r>}(I^{\vee})^{\ast})$$
\end{corollary}

\vskip 2mm

It follows that one may think Lyubeznik numbers
of a squarefree monomial $I$ as a measure of the
acyclicity of the $r$-linear strand of the Alexander dual
$I^{\vee}$.

\vskip 2mm 

\begin{remark}
As a summary of the dictionary between local cohomology modules and free resolutions we have:

\begin{itemize}

\item  The graded pieces  $[H^{r }_I(R)]_{\alpha}$ correspond to the Betti numbers $\beta_{|\alpha | -r,\alpha}(I^{\vee})$

\item  The $n$-hypercube of $H_I^r(R)$ corresponds to the $r$-linear strand $\mathbb{L}_{\bullet}^{<r>}(I^{\vee}) $

\end{itemize}

\vskip 2mm 

Given a free resolution $\mathbb{L}_{\bullet}$ of a finitely generated graded $R$-module $M$, 
D.~Eisenbud, G.~Fl{\o}ystad and F.O.~Schreyer \cite{EFS} defined its {\it linear part} 
as the complex  ${\rm lin}(\mathbb{L}_{\bullet})$ obtained by 
erasing the terms of degree $\geq 2$ from the matrices of the differential maps. To measure the acyclicity
of the linear part, J.~Herzog and S.~Iyengar \cite{HI} introduced the {\it linearity defect} of $M$ as 
${\rm ld}_R(M):=\sup \{p \hskip 2mm | \hskip 2mm H_p({\rm lin}(\mathbb{L}_{\bullet}))\}$.
Therefore we also have:

\vskip 2mm

\begin{itemize}

\item The $n$-hypercubes of $H_I^r(R)$, $\forall r$ correspond to the linear part  ${\rm lin}(\mathbb{L}_{\bullet}(I^{\vee})) $

\item The Lyubeznik table of $R/I$ can be viewed as a generalization of ${\rm ld}_R(I^{\vee})$

\end{itemize}

\end{remark}

\subsection{Examples}

It is well-known that Cohen-Macaulay squarefree monomial
ideals have a trivial Lyubeznik table
$$\Lambda(R/I)  = \left(
                    \begin{array}{ccc}
                      0 & \cdots & 0 \\
                       & \ddots & \vdots \\
                       &  & 1 \\
                    \end{array}
                  \right)
$$ because they only have one non-vanishing local cohomology module.
Recall that its Alexander dual has a linear resolution (see
\cite{ER98}) so its acyclic. In general, there are
non-Cohen-Macaulay ideals with trivial Lyubeznik table. Some of them
are far from having only one local cohomology module different from
zero.

\vskip 2mm

\begin{example}
Consider the ideal in $k[x_1,\dots,x_9]$:

\vskip 2mm

$I=(x_1,x_2)\cap(x_3,x_4)\cap (x_5,x_6) \cap (x_7,x_8)\cap (x_9,x_1)\cap (x_9,x_2)\cap (x_9,x_3)\cap (x_9,x_4)\cap (x_9,x_5)\cap $

$\hskip .8cm  \cap (x_9,x_6)\cap (x_9,x_7)\cap (x_9,x_8)$

\vskip 2mm

\noindent The non-vanishing local cohomology modules are $H^r_I(R)$
, $r=2,3,4,5$ but the Lyubeznik table is trivial.

\end{example}

\vskip 2mm One may characterize ideals with trivial Lyubeznik table
using a weaker condition than being Cohen-Macaulay,  the class of
{\it sequentially Cohen-Macaulay} ideals given by R.~Stanley
\cite{Sta96}. J.~Herzog and T.~Hibi \cite{HH} introduced the class
of {\it componentwise linear} ideals and proved that their Alexander
dual are sequentially Cohen-Macaulay. The following result is a
direct consequence of \cite[Prop. 4.9]{Ya00}, \cite[Thm. 3.2.8]{Rom01} where componentwise
linear ideals are characterized as those having acyclic linear
strands.

\begin{proposition}
Let $I\subseteq R=k[x_1,\dots,x_n]$ be a squarefree monomial ideal.
Then, the following conditions are equivalent:
\begin{itemize}
\item[i)] $R/I$ is sequentially Cohen-Macaulay.
\item[ii)] $R/I$ has a trivial Lyubeznik table.
\end{itemize}

\end{proposition}

\vskip 2mm

The simplest examples of ideals with non-trivial Lyubeznik table are
minimal non-Cohen-Macaulay squarefree monomial ideals (see
\cite{Ly})

\begin{example}
The unique minimal non-Cohen-Macaulay squarefree monomial ideal of
pure height two in $R=k[x_1,\dots,x_n]$ is:
$$\fa_n= (x_1, x_3)\cap\cdots\cap(x_1, x_{n- 1})\cap(x_2, x_4)\cap\cdots\cap(x_2, x_n)\cap(x_3, x_5)\cap\cdots\cap(x_{n- 2}, x_n).$$

\vskip 2mm

$\bullet $ $\fa_4= (x_1, x_3)\cap(x_2, x_4)$.

\vskip 2mm \noindent  We have $H_{\fa_4}^{2}(R) \cong H_{(x_1,
x_3)}^{2}(R) \oplus H_{(x_2, x_4)}^{2}(R)$ and $H_{\fa_4}^{3}(R) \cong
E_{\bf 1}.$ Thus its Lyubeznik table is
$$\Lambda(R/{\fa_4})  = \begin{pmatrix}
   0 & 1 & 0 \\
     & 0 & 0 \\
     &   & 2
\end{pmatrix}$$

$\bullet $ $\fa_5= (x_1, x_3)\cap (x_1, x_4)\cap (x_2, x_4)\cap
(x_2, x_5)\cap (x_3, x_5)$.

\vskip 2mm \noindent We have $H_{\fa_5}^{3}(R) \cong E_{\bf 1}$ and the
hypercube associated to $H_{\fa_5}^{2}(R)$ satisfy $[H_{\fa_5}^{2}(R)]_\alpha
\cong k$ for

\vskip 2mm

$\cdot$ $\alpha=
(1,0,1,0,0),(1,0,0,1,0),(0,1,0,1,0),(0,1,0,0,1),(0,0,1,0,1)$

$\cdot$ $\alpha=
(1,1,0,1,0),(1,0,1,1,0),(1,0,1,0,1),(0,1,1,0,1),(0,1,0,1,1)$

\vskip 2mm \noindent The complex associated to the hypercube is
$${\xymatrix { 0 & 0 \ar[l] & 0 \ar[l] & k^5 \ar[l] & k^5
\ar[l]_{u_2} & 0 \ar[l] &0 \ar[l] & 0 \ar[l]}}$$ where the matrix
corresponding to $u_2$ is the rank $4$ matrix:

{\tiny $$ \begin{pmatrix}
  0 & -1 & -1 & 0  & 0 \\
   1 & -1 & 0 &  0 & 0 \\
  -1  & 0  & 0  & 0 & 1 \\
   0 &  0 & 0 & 1 & 1 \\
   0 & 0 & -1  & -1  & 0
\end{pmatrix}$$} Thus its Lyubeznik table is
$$\Lambda(R/{\fa_5})  = \begin{pmatrix}
  0& 0 & 1 & 0 \\
    &0 & 0 & 0 \\
     &&  0 & 1 \\
     &&& 1
\end{pmatrix}$$ One should notice that $H_{\fa_5}^{2}(R)$ is irreducible since all the extension problems
associated to it are non-trivial.

\begin{remark} In general one gets
{ $$\Lambda(R/{\fa_n})  = \begin{pmatrix}
  0 & 0 & 0 & \cdots & 0 & 1 & 0 \\
    & 0 & 0 & \cdots & 0 & 0 & 0 \\
    &   & 0 & & 0 & 0 & 1 \\
    &   & & \ddots &  & 0 & 0 \\
&   &   & &  & \vdots & \vdots \\
    &   & &  &   & 0 & 0 \\
    &   & &  &   &   & 1
\end{pmatrix}$$} and the result agrees with \cite[Cor. 5.5]{Sch}

\end{remark}

\end{example}

\vskip 2mm

It is well-know that local cohomology modules as well as free
resolutions depend on the characteristic of the base field,  the
most recurrent example being the Stanley-Reisner ideal associated to
a minimal triangulation of $\mathbb{P}_{\bR}^2$. Thus, Lyubeznik
numbers also depend on the characteristic.

\vskip 2mm

\begin{example} Consider the ideal in $R=k[x_1,\dots, x_6]$:

\vskip 2mm

$I=(x_1x_2x_3,x_1x_2x_4,x_1x_3x_5,x_2x_4x_5,x_3x_4x_5,x_2x_3x_6,x_1x_4x_6,x_3x_4x_6,x_1x_5x_6,x_2x_5x_6)$

\vskip 2mm

\noindent The Lyubeznik table in characteristic zero and two are respectively:
$$\Lambda_{\bQ}(R/I)  = \begin{pmatrix}
   0 & 0 & 0 & 0 \\
     & 0 & 0  & 0\\
     &   & 0 & 0 \\
&   &  & 1
\end{pmatrix} \hskip 5mm \Lambda_{\bZ/2\bZ}(R/I)  = \begin{pmatrix}
   0 & 0 & 1 & 0 \\
     & 0 & 0  & 0\\
     &   & 0 & 1 \\
&   &  & 1
\end{pmatrix}$$

 \end{example}

\section{Injective dimension of local cohomology modules}

Let $(R,\fM,k)$ be a local  ring and let $M$ be an $R$-module. The
small support of $M$ introduced by H.~B.~Foxby \cite{Fox} is defined
as
$${\rm supp}_R M:= \{\fp \in {\rm Spec} R \hskip 2mm | \hskip 2mm {\rm depth}_{R_{\fp}} M_{\fp} < \infty\},$$
where ${\rm depth}_R M := {\rm inf} \{i \in \bZ \hskip 2mm | \hskip 2mm {\rm Ext}_R^i(R/\fM, M)\neq 0\}= {\rm inf} \{i \in \bZ \hskip 2mm | \hskip 2mm \mu_i(\fM, M)\neq 0\}$.
In terms of Bass numbers we have that $\fp \in {\rm supp}_R M$ if and only if there exists some integer $i\geq 0$
such that ${\mu}_i(\fp, M) \neq 0$. It is also worth to point out that ${\supp}_R M \subseteq {\Supp}_R M$, and
equality holds when $M$ is finitely generated.

\vskip 2mm

Bass numbers of finitely generated modules are known to satisfy the following properties:

\begin{itemize}
 \item [1)] ${\mu}_i(\fp, M) < +\infty$, $\forall i$, $\forall \fp \in {\Supp}_R M$

\item [2)] Let $\fp \subseteq \fQ \in {\rm Spec} R$ such that $\hlt (\fQ/\fp)=s$. Then $${\mu}_i(\fp, M)\neq 0 \Longrightarrow \mu_{i+s}(\fQ, M)\neq 0.$$

 \item [3)] ${\rm id}_R M := {\rm sup} \{i \in \bZ \hskip 2mm | \hskip 2mm \mu_i(\fM, M)\neq 0\}$

 \item [4)] ${\rm depth}_R M \leq {\dim}_R M \leq {\rm id}_R M$

\end{itemize}

\vskip 2mm

When $M$ is not finitely generated, similar properties for Bass
numbers are known for some special cases. A.~M.~Simon \cite{Si}
proved that properties $2)$ and $3)$ are still true for complete
modules and M.~Hellus \cite{He08} proved that ${\dim}_R M \leq {\rm
id}_R M$ for cofinite modules.

\vskip 2mm

For the case of local cohomology modules, C.~Huneke and R.~Sharp
\cite{HS93} and G.~Lyubeznik \cite{Ly93, Ly97}, proved that for a
regular local ring $(R,\fM,k)$ containing a field $k$:

\begin{itemize}
 \item [1)] ${\mu}_i(\fp, H_I^r(R)) < +\infty$, $\forall i$, $\forall r$, $\forall \fp \in {\Supp}_R H_I^r(R)$

 \item [4')] $ {\rm id}_R H_I^r(R) \leq {\dim}_R H_I^r(R)$

\end{itemize}

\vskip 2mm

In this Section we want to study property $2)$
for the particular case of local cohomology modules supported on monomial ideals
and give a sharper bound to $4')$ in terms of the small support. We start with the
following well-known general result on the minimal primes in the support of local cohomology modules.

\begin{proposition}
Let $(R,\fM)$ be a regular local ring containing a field $k$, $I
\subseteq R$ be any ideal and $\fp \in {\Supp}_R \hskip 1mm
H^r_I(R)$ be a minimal prime. Then we have $\mu_0(\fp, H^r_I(R))\neq
0$, $\mu_i(\fp, H^r_I(R))=0$ $\forall i>0$.

\end{proposition}

\begin{proof}
 dim$H^r_I(R)_{\fp}=0$ so  $H^r_I(R)_{\fp} \cong E(R_{\fp}/\fp R_{\fp})^{\mu_0(\fp, H^r_I(R))}$ by \cite[Thm 3.4]{Ly93}
\end{proof}

\begin{corollary}
Let $(R,\fM,k)$ be a regular local ring containing a field $k$ and
$I \subseteq R$ be any ideal.
 If $\fp \in {\Supp}_R \hskip 1mm H^r_I(R)$ is minimal then $\fp \in {\supp}_R H^r_I(R)$. Thus, ${\Supp}_R \hskip 1mm H^r_I(R)$ and ${\supp}_R \hskip 1mm H^r_I(R)$ have the same minimal primes.

\end{corollary}

The converse statement in Proposition $5.1$ does not hold true.

\begin{example}
Consider the monomial ideal $I=(x_1,x_2,x_5)\cap (x_3,x_4,x_5)\cap
(x_1,x_2,x_3,x_4)$. The support of the corresponding local
cohomology modules are:


\vskip 2mm

\hskip 1cm ${\Supp}_R \hskip 1mm H_{I}^3(R)= V(x_1,x_2,x_5) \cup V(x_3,x_4,x_5).$

\hskip 1cm ${\Supp}_R \hskip 1mm H_{I}^4(R) = V(x_1,x_2,x_3,x_4).$

\vskip 2mm

\noindent The Bass numbers of $H_{I}^3(R)$ and $H_{I}^4(R)$ are respectively

$$ \begin{tabular}{|c|c|c|c|c|}\hline
 $\fp_\alpha$ & $\mu_0$ & $\mu_1$ & $\mu_2$ \\
 \hline $(x_1,x_2,x_5)$ & $1$ & - & - \\
  $(x_3,x_4,x_5)$ & $1$ & - & - \\
  $(x_1,x_2,x_i,x_5)$  & - & $1$ & - \\
  $(x_i,x_3,x_4,x_5)$  & - & $1$ & - \\
  $(x_1,x_2,x_3,x_4,x_5)$ & - & - & $2$ \\ \hline
\end{tabular}   \hskip 1cm
\begin{tabular}{|c|c|c|c|}\hline
 $\fp_\alpha$ & $\mu_0$ & $\mu_1$ & $\mu_2$ \\
 \hline $(x_1,x_2,x_3,x_4)$ & $1$ & - & - \\
  $(x_1,x_2,x_3,x_4,x_5)$ & $1$ & - & - \\
   \hline
\end{tabular}
$$

\vskip 2mm \noindent In particular, its Lyubeznik table is
$$\Lambda(R/I) =
\begin{pmatrix}
 0 & 1 & 0 \\
    &  0 & 0 \\
     && 2
\end{pmatrix}$$

\noindent Notice that  $\fM=(x_1,x_2,x_3,x_4,x_5)$ is not a minimal prime in the support of $H_{I}^4(R)$ but $\mu_0(\fM, H^4_I(R))\neq 0$, $\mu_i(\fM, H^4_I(R))=0$ $\forall i>0$. We have to point out that this module is not irreducible\footnote{It is enough to check out the corresponding $n$-hypercubes.} $$H^4_I(R) \cong
E_{(1,1,1,1,0)} \oplus E_{(1,1,1,1,1)}$$

\end{example}

From now on we will stick to the case of local cohomology modules
supported on squarefree monomial ideals. The methods developed in
the previous Sections allow us to describe the Bass numbers in the
minimal $^\ast$injective resolution of a module with variation zero
$M$. That is:
$$\mathbb{I}^{\bullet}(M): \hskip 3mm \xymatrix{ 0 \ar[r]& I^{0}
\ar[r]^{d^{0}}&  I^1 \ar[r]^{d^1}&\cdots \ar[r]^{d^{m-1}}& I^{m}
\ar[r]^{d^m}& \cdots},$$ where the $j$-th term is $$I^j =
\bigoplus_{\alpha \in {\{0,1\}^n}}
E_{\alpha}^{\mu_{j}(\fp_\alpha,M)}= \bigoplus_{\alpha \in
{\{0,1\}^n}} {^{\ast}E(R/\fp_{\alpha})}(\bf 1
)^{\mu_{j}(\fp_\alpha,M)} .$$ In particular we are able to compute
the injective dimension of $M$ in the category of $\bZ^n$-graded
$R$-modules that we  denote ${^\ast}{\rm id}_R M$. We can also
define the $\bZ^n$-graded small support that we denote
${^\ast}\supp_R M$ as the set of face ideals in the support of $M$
that at least have a Bass number different from zero.

\vskip 2mm

If we want to compute the Bass numbers with respect to any prime
ideal, the injective dimension of $M$ as $R$-module and the small
support we have to refer to the result of S.~Goto and K.~I.~Watanabe
\cite[Thm. 1.2.3]{GW78}. Namely, given any prime ideal $\fp \in
\spec R$, let $\fp_\alpha$ be the largest face ideal contained in
$\fp$. If $\hlt(\fp/\fp_\alpha)=s$ then $\mu_p(\fp_\alpha,M)=
\mu_{p+s}(\fp,M)$. Notice that in general we have ${^\ast}{\rm id}_R
M \leq {\rm id}_R M$.

\vskip 2mm

To compare the injective dimension and the dimension of a local
cohomology module $M=H_I^r(R)$  we are going to consider chains of
prime face ideals $\fp_0 \subseteq \fp_1 \subseteq \cdots \subseteq
\fM$ in the support of  $M$ such that $\fp_0$ is minimal. The Bass
numbers with respect to $\fp_0$ are completely determined and, even
though property $2)$ is no longer true, we have some control on the
Bass numbers of $\fp_i$ depending on the structure of the
corresponding $n$-hypercube. For simplicity, assume that $\fp_i$ is
a face ideal $\fp_\alpha \subseteq \fM$ of height $n-1$ and $x_n\in
\fM \setminus \fp_\alpha$ and that the Bass numbers with respect to
$\fp_\alpha$ are known. We are going to compute the Bass numbers
with respect to $\fM$ using the degree {\bf 1} part of the exact
sequence of \v{C}ech complexes
$$ 0 \longrightarrow
{\check{C}}^{\bullet}_{\fp_\alpha}(M_{x_{n}})[-1] \longrightarrow
{\check{C}}^{\bullet}_{\fM}(M) \longrightarrow
{\check{C}}^{\bullet}_{\fp_\alpha}(M)\lra 0$$

Let $\cM^{\bullet}$ be the complex associated to the $n$-hypercube
of $M$ that is isomorphic to $[{\check{C}}^{\bullet}_{\fM}(M)]_{\bf
1}$. For any $\beta \in \{0,1\}^n$, let $\cM^{\bullet}_{\leq \beta}$
(resp. $\cM^{\bullet}_{\geq \beta}$) be the subcomplex of
$\cM^{\bullet}$ with pieces of degree $\leq \beta$ (resp. $\geq
\beta$). Using the techniques of Section $3$ one may see that
$$ 0 \longleftarrow
[{\check{C}}^{\bullet}_{\fp_\alpha}(M_{x_{n}})[-1]]_{\bf 1}
\longleftarrow [{\check{C}}^{\bullet}_{\fM}(M)]_{\bf 1}
\longleftarrow [{\check{C}}^{\bullet}_{\fp_\alpha}(M)]_{\bf
1}\longleftarrow 0$$ is isomorphic to the short exact sequence
$$ 0 \longleftarrow
\cM_{\leq \alpha}^{\bullet} \longleftarrow \cM^{\bullet} \longleftarrow \cM_{\geq {\bf
1}-\alpha}^{\bullet} \longleftarrow 0$$

\begin{example} The short exact sequence $ 0 \longleftarrow
\cM_{\leq (1,1,0)}^{\bullet} \longleftarrow \cM^{\bullet}
\longleftarrow \cM_{\geq (0,0,1)}^{\bullet} \longleftarrow 0$ can be
visualized from the corresponding $3$-hypercube as follows:

\vskip 3mm

{\tiny $
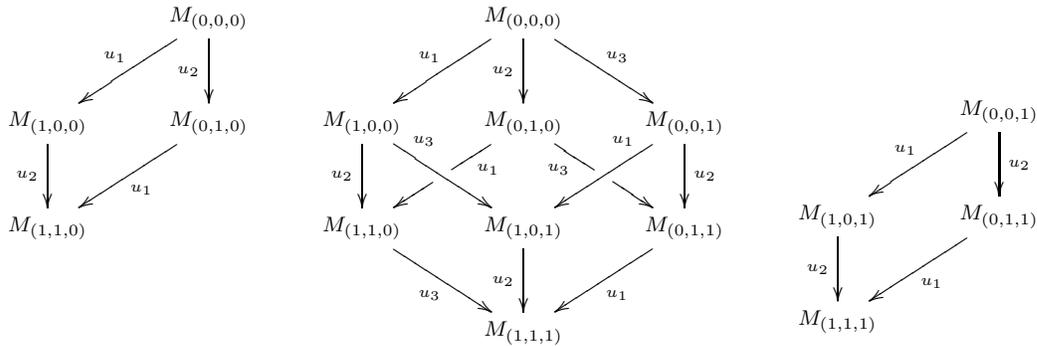
{\xymatrix { &M_{(0,0,0)} \ar[dl]_{u_1} \ar[d]_{u_2} \\
M_{(1,0,0)} \ar[d]_{u_2} & M_{(0,1,0)} \ar[dl]^(.5){u_1}      \\
M_{(1,1,0)}  &  }}$  } \hskip 5mm {\tiny ${\xymatrix { &M_{(0,0,0)}
\ar[dl]_{u_1} \ar[d]_{u_2} \ar[dr]^{u_3} &
\\ M_{(1,0,0)} \ar[d]_{u_2} \ar[dr]^(.3){u_3}  & M_{(0,1,0)}\ar[dr]_(.3){u_3}|\hole \ar[dl]^(.3){u_1}|\hole & M_{(0,0,1)}  \ar[dl]_(.3){u_1} \ar[d]^{u_2}
\\ M_{(1,1,0)} \ar[dr]_{u_3}& M_{(1,0,1)} \ar[d]_{u_2} & M_{(0,1,1)} \ar[dl]^{u_1}
\\& M_{(1,1,1)} & }} $ }  \hskip 5mm
{\tiny ${\xymatrix {  & \\   &
M_{(0,0,1)}  \ar[dl]_(.5){u_1} \ar[d]^{u_2}\\
M_{(1,0,1)}\ar[d]_{u_2} & M_{(0,1,1)} \ar[dl]^{u_1} \\
M_{(1,1,1)} & }} $ }

\end{example}

At this point we should notice the following key observations that
we will use throughout this Section:

\vskip 3mm

\begin{itemize}

\item [i)] We have $\cM^{\bullet}_{\leq
\alpha}\cong
[{\check{C}}^{\bullet}_{\fp_\alpha}(M_{x_{n}})[-1]]_{\bf 1} \cong
[{\check{C}}^{\bullet}_{\fp_\alpha}(M)]_{\alpha}$, thus
$\mu_p(\fp_\alpha, M)= {\rm dim}_k H_p(\cM^{\bullet}_{\leq
\alpha})$.

\vskip 2mm

\item [ii)] Consider the long exact sequence $$  \cdots \lra
H^{p-1}_{\fp_\alpha}(M_{x_n}) \longrightarrow H^p_{\fM}(M)
\longrightarrow H^p_{\fp_\alpha}(M)\stackrel{\delta^p}\lra
H^p_{\fp_\alpha}(M_{x_n}) \longrightarrow H^{p+1}_{\fM}(M)
\longrightarrow \cdots$$ associated to the short exact sequence of
\v{C}ech complexes. Its degree {\bf 1} part is
$$
\cdots \longleftarrow H_{p-1}(\cM_{\leq \alpha}^{\bullet})
\longleftarrow H_p(\cM^{\bullet}) \longleftarrow H_p(\cM_{\geq {\bf
1}-\alpha}^{\bullet})\stackrel{\delta^p}\longleftarrow H_p(\cM_{\leq
\alpha}^{\bullet}) \longleftarrow H_{p+1}(\cM^{\bullet})
\longleftarrow \cdots$$

\noindent but it might be useful to view it as
$$ \cdots \longleftarrow
k^{\mu_{p-1}(\fp_\alpha, M)} \longleftarrow [H^p_{\fM}(M)]_{\bf 1}
\longleftarrow [H^p_{\fp_\alpha}(M)]_{\bf 1}
\stackrel{\delta^p}\longleftarrow k^{\mu_{p}(\fp_\alpha, M)}
\longleftarrow [H^{p+1}_{\fM}(M)]_{\bf 1} \longleftarrow \cdots$$

\noindent or even as the complex $$ \cdots \longleftarrow
k^{\mu_{p-1}(\fp_\alpha, M)} \longleftarrow k^{\mu_{p}(\fM, M)}
\longleftarrow [H^p_{\fp_\alpha}(M)]_{\bf 1}
\stackrel{\delta^p}\longleftarrow k^{\mu_{p}(\fp_\alpha, M)}
\longleftarrow k^{\mu_{p+1}(\fM, M)} \longleftarrow \cdots$$

\vskip 2mm \noindent  Notice that the connecting morphisms
$\delta^p$ are the classes, in the corresponding homology groups, of
the canonical morphisms $u_{\alpha,n}$ that describe the
$n$-hypercube of $M$.

\vskip 2mm

\item [iii)] The 'difference' between $\mu_p(\fp_\alpha, M)$ and $\mu_{p+1}(\fM,
M)$, i.e. the 'difference' between $H_p(\cM^{\bullet}_{\leq
\alpha})$ and $H_{p+1}(\cM^{\bullet})$, comes from the homology of
the complex $\cM^{\bullet}_{\geq {\bf 1}-\alpha}$. Roughly speaking,
it comes from the contribution of other chains of prime face ideals
$\fQ_0 \subseteq \fQ_1 \subseteq \cdots \subseteq \fM$ in the
support of a local cohomology module $M=H_I^r(R)$ such that $\fQ_0$
is minimal and not containing $\fp_\alpha$\footnote{We use the fact
that for local cohomology modules $M=H_I^r(R)$ the degree {\bf 0}
part of the $n$-hypercube is always zero, i.e. its minimal primes
have height $>0$}.

\end{itemize}

\vskip 2mm

{\bf Discussion 1:} Consider the case  where $\fp_\alpha \subseteq
\fM$ is a minimal prime of height $n-1$ in the support of $M$. We
have
$$ 0 \longleftarrow [H^0_{\fM}(M)]_{\bf 1} \longleftarrow
[H^0_{\fp_\alpha}(M)]_{\bf 1}\stackrel{\delta^0}\longleftarrow k
\longleftarrow [H^{1}_{\fM}(M)]_{\bf 1} \longleftarrow
[H^1_{\fp_\alpha}(M)]_{\bf 1} \longleftarrow 0$$ and
$[H^{i}_{\fM}(M)]_{\bf 1}\cong [H^i_{\fp_\alpha}(M)]_{\bf 1}$, for
all $i\geq 2$, so $\fp_\alpha$ contributes to $\mu_{p}(\fM, M)$ for
$p=0,1$. In particular, we have:

\vskip 2mm

\begin{itemize}

\item [$\cdot$] $\mu_{0}(\fM, M)=0$ if and only if $[H^0_{\fp_\alpha}(M)]_{\bf
1}=0$ or $[H^0_{\fp_\alpha}(M)]_{\bf 1}=k$ and $\delta^0 \neq 0$.

\item [$\cdot$] $\mu_{1}(\fM, M)=0$ if and only if $[H^1_{\fp_\alpha}(M)]_{\bf
1}=0$ and $\delta^0 \neq 0$.

\end{itemize}

\vskip 2mm

The non-vanishing of the $0$-th Bass number is related to the
decomposability of the local cohomology module. One should compare
the following result with Prop $5.1$ and check out the local
cohomology module $H^4_I(R)$ in Example $5.3$.

\begin{proposition}
Let $\fp_{\alpha} \in {\Supp}_R \hskip 1mm H^r_I(R)$ be a prime
ideal such that $\mu_0(\fp_\alpha, H^r_I(R))\neq 0$. Then,
$E_\alpha^{\mu_0(\fp_\alpha, H^r_I(R))}$ is a direct summand of
$H^r_I(R)_{\fp_\alpha}$.
\end{proposition}

\begin{proof}
We assume that $\fp_\alpha=\fM$ and we denote $\mu_0:=\mu_0(\fM,
H^r_I(R))$. The first terms of the complex $\cM^\bullet$ associated
to the $n$-hypercube of $H^r_I(R)$ have the form $0\longleftarrow
k^{m_0}\stackrel{u_{0}}\longleftarrow k^{m_1}$ with $\mu_0=m_0 -
{\rm rk} \hskip 1mm u_0 > 0$. The linear map
$u_0=\oplus_{|\alpha|=n-1} u_{\alpha,i}$ determine the extension
classes of the short exact sequence $0 \lra F_{n-1}\lra H^r_I(R)\lra
E_{\bf 1}^{m_0} \lra 0$ associated to the filtration $\{
{F_j}\}_{0\le j\le n}$ of $H^r_I(R)$. Thus, we have a decomposition
$H^r_I(R)\cong E_{\bf 1}^{\mu_0} \oplus M$, where $M$ corresponds to
the extension $0 \lra F_{n-1}\lra M \lra E_{\bf 1}^{{\rm rk}u_0}
\lra 0$.
\end{proof}


\vskip 2mm

{\bf Discussion 2:} In general, let $s = {\rm max}\{ i\in \bZ_{\geq
0} \hskip 2mm | \hskip 2mm {\mu}_i(\fp_\alpha, M)\neq 0\}$, then we
have  $$  \cdots \longleftarrow [H^s_{\fM}(M)]_{\bf 1}
\longleftarrow [H^s_{\fp_\alpha}(M)]_{\bf
1}\stackrel{\delta^s}\longleftarrow k^{\mu_s(\fp_\alpha,M)}
\longleftarrow [H^{s+1}_{\fM}(M)]_{\bf 1} \longleftarrow
[H^{s+1}_{\fp_\alpha}(M)]_{\bf 1} \longleftarrow 0$$ and
$[H^{i}_{\fM}(M)]_{\bf 1}\cong [H^i_{\fp_\alpha}(M)]_{\bf 1}$, for
all $i\geq s+2$, so $\fp_\alpha$ contributes to $\mu_{p}(\fM, M)$
for $p\leq s+1$. Again, we can describe conditions for the vanishing
of $\mu_{s}(\fM, M)$ and $\mu_{s+1}(\fM, M)$ in terms of the
connecting morphism $\delta^s$. One can find examples where any
situation is possible.

\vskip 2mm

\begin{itemize}

\item [$\cdot$] The local cohomology module $H_I^3(R)$ in Example $5.3$ satisfies
for any $\fp_\alpha \subseteq \fM$ such that
$\hlt(\fM/\fp_\alpha)=1$,  $\mu_{s}(\fp_\alpha, M)\neq 0$, $
\mu_{s+1}(\fM, M)\neq 0$ and $\mu_{s}(\fM, M)= 0$ for $s=1$.

\vskip 2mm

\item [$\cdot$] The local cohomology module $H_I^4(R)$ in Example $5.3$ satisfies
$\mu_{s}(\fp_{(1,1,1,1,0)}, M)= \mu_{s}(\fM, M)=1$ and
$\mu_{s+1}(\fM, M)= 0$ for $s=0$.

\vskip 2mm

\item [$\cdot$] The local cohomology module $H_I^3(R)$ in Example $5.7$ satisfies
$\mu_{s}(\fp_{(1,1,1,0,0)}, M)= 1$ and $\mu_{s}(\fp_{(1,1,1,1,0)},
M)=\mu_{s+1}(\fp_{(1,1,1,1,0)}, M)= 0$ for $s=0$.

\vskip 2mm

\item [$\cdot$] The local cohomology module $H_{\fa_5}^2(R)$ in Example $4.4$ satisfies
$\mu_{s}(\fp_{(1,1,1,0,1)}, M)= \mu_{s}(\fM, M) = \mu_{s+1}(\fM,
M)=1$  for $s=2$.

\end{itemize}

\vskip 2mm

\begin{remark}

One might be tempted to think that the condition
$\mu_{s}(\fp_\alpha, M)\neq 0$ and $\mu_{s}(\fM, M)\neq 0$ is
related to the decomposability of the corresponding module $M$. This
is not the case as it shows Example $3.5$ where we have a short
exact sequence $$0 \longleftarrow [H^2_{\fM}(M)]_{\bf 1}
\longleftarrow [H^2_{\fp_{(1,1,1,0,1)}}(M)]_{\bf 1} \cong k
\stackrel{\delta^2}\longleftarrow k \longleftarrow
[H^3_{\fM}(M)]_{\bf 1} \longleftarrow 0$$ where the connecting
morphism $\delta^2$ is zero even though the local cohomology module
$H_{\fa_5}^2(R)$ is indecomposable, i.e. the canonical morphisms
$u_{\alpha,i}$ are not trivial but their classes in homology make
the connecting morphism trivial.
\end{remark}

\vskip 2mm

{\bf Discussion 3:} In the case that there exists a prime ideal
$\fp_\alpha \not \in {\rm supp}_R M$, then we have
$[H^{i}_{\fM}(M)]_{\bf 1}\cong [H^i_{\fp_\alpha}(M)]_{\bf 1}$, for
all $i$. Therefore, the contribution to the Bass numbers
$\mu_{p}(\fM, M)$ comes from other chains of prime face ideals
$\fQ_0 \subseteq \fQ_1 \subseteq \cdots \subseteq \fM$. This is what
happens in the following example:

\begin{example}
Consider the ideal $I=(x_1,x_4)\cap (x_2,x_5)\cap (x_1,x_2,x_3)$.
The non-vanishing pieces of the hypercube associated to  the
corresponding local cohomology modules are:

\vskip 2mm

 $[H_{I}^2(R)]_\alpha= k$  \hskip 2mm for $\alpha=(1,0,0,1,0),(0,1,0,0,1)$.

 $[H_{I}^3(R)]_\alpha= k$  \hskip 2mm for $\alpha=(1,1,1,0,0),(1,1,1,1,0),(1,1,1,0,1),(1,1,0,1,1),(1,1,1,1,1)$.

\vskip 2mm

\noindent Notice that $H_{I}^2(R) \cong H_{(x_1,x_4)}^2(R) \oplus H_{(x_2,x_5)}^2(R)$ and the complex associated to the hypercube of  $H_{I}^3(R)$ is:
$${\xymatrix { 0 & k \ar[l]  && k^3 \ar[ll]_{{\Tiny \begin{pmatrix}
 1 & 1 & 1  \end{pmatrix}}} && k \ar[ll]_{{\Tiny
\left(\begin{array}{rrr}
-1\\ 1\\ 0\\  \end{array}\right)}}
 & 0
\ar[l] }}$$



\noindent Thus, the Bass numbers of $H_{I}^2(R)$ and $H_{I}^3(R)$ are
respectively

 $$
\begin{tabular}{|c|c|c|c|c|}\hline
 $\fp_\alpha$ & $\mu_0$ & $\mu_1$ & $\mu_2$ & $\mu_3$\\
 \hline $(x_1,x_4)$ & $1$ & - & - & -\\
  $(x_2,x_5)$ & $1$ & - & - & -\\
  $(x_1,x_4,x_i)$  & - & $1$ & - & -\\
  $(x_2,x_5,x_i)$  & - & $1$ & - & -\\
  $(x_1,x_4,x_i,x_j)$  & - &-& $1$  & -\\
  $(x_2,x_5,x_i,x_j)$ & - &-& $1$  & -\\
  $(x_1,x_2,x_3,x_4,x_5)$ & - & - & - & $2$ \\ \hline
\end{tabular},  \hskip 1cm \begin{tabular}{|c|c|c|c|}\hline
 $\fp_\alpha$ & $\mu_0$ & $\mu_1$ & $\mu_2$  \\ \hline
 $(x_1,x_2,x_3)$ & $1$ & - & -  \\
$(x_1,x_2,x_3,x_4)$ & - & - & -  \\ $(x_1,x_2,x_3,x_5)$ & - & - & -
\\ $(x_1,x_2,x_4,x_5)$ & $1$ & - & -  \\
 $(x_1,x_2,x_3,x_4,x_5)$ & - & $1$ & -  \\ \hline
\end{tabular}
$$

\vskip 2mm

\noindent Its Lyubeznik table is:
$$\Lambda(R/I)  = \begin{pmatrix}
  0& 0 & 0 & 0 \\
    &0 & 1 & 0 \\
     &&  0 & 0 \\
     &&& 2
\end{pmatrix}$$

\vskip 2mm

\noindent We also have:

\vskip 2mm

\hskip 1cm $\cdot$ ${^\ast}\id_R H_I^{2}(R) = \dm H_I^{2}(R)=3.$
\hskip 1.7cm $\cdot$ $1={^\ast}\id_R H_I^{3}(R) < \dm H_I^{3}(R)=2.$

\hskip 1cm $\cdot$ $ \Min_R (H_I^{2}(R)) = \Ass_R (H_I^{2}(R)).$
\hskip 1.4cm $\cdot$ $ \Min_R (H_I^{3}(R)) = \Ass_R (H_I^{3}(R)).$

\hskip 1cm $\cdot$ $ \supp_R (H_I^{2}(R)) = \Supp_R (H_I^{2}(R)).$
\hskip 1cm $\cdot$ $ \supp_R (H_I^{3}(R))  \subset \Supp_R
(H_I^{3}(R)).$

\vskip 2mm

\noindent In particular $(x_1,x_2,x_3,x_4)$ and $(x_1,x_2,x_3,x_5) $
do not belong to $\supp_R (H_{I}^3(R))$.

\end{example}

\vskip 2mm

It follows from the previous discussions that the length of the
injective resolution of the local cohomology module $H^r_I(R)$ has a
controlled growth when we consider chains of prime face ideals
$\fp_0 \subseteq \fp_1 \subseteq \cdots \subseteq \fM$ starting with
a minimal prime ideal $\fp_0$.


\begin{proposition}
Let $I \subseteq R=k[x_1,\dots,x_n]$ be a squarefree monomial ideal
and set $$s := {\rm max}\{ i\in \bZ_{\geq 0} \hskip 2mm | \hskip 2mm
{\mu}_i(\fp_\alpha, H^r_I(R))\neq 0\}$$ for all prime ideals
$\fp_{\alpha} \in {\Supp}_R \hskip 1mm H^r_I(R)$ such that
$|\alpha|=n-1$. Then $\mu_{t}(\fM, H^r_I(R))= 0$ $\forall
 t> s+1$.
\end{proposition}

Therefore we get the main result of this Section:

\begin{theorem}
Let $I \subseteq R=k[x_1,\dots,x_n]$ be a squarefree monomial ideal.
Then, $\forall r$ we have $$ {^\ast}{\rm id}_R H_I^r(R) \leq
{\dim}_R {^\ast}\supp_R H_I^r(R)$$
\end{theorem}

\begin{remark}
Using \cite[Thm. 1.2.3]{GW78} we also have $$ {\rm id}_R H_I^r(R)
\leq {\dim}_R \hskip 1mm\supp_R H_I^r(R)$$ but one must be careful with the ring $R$ we consider. 
In the example above we have:

\begin{itemize}
 \item [$\cdot$] ${^\ast}\id_R H_I^{3}(R) = \id_R H_I^{3}(R) < {\dim}_R \hskip 1mm\supp_R H_I^3(R)$ if $R=k[[x_1,\dots,x_n]]$

\item [$\cdot$] ${^\ast}\id_R H_I^{3}(R) < \id_R H_I^{3}(R)= {\dim}_R \hskip 1mm\supp_R H_I^3(R)$ if $R=k[x_1,\dots,x_n]$.

\end{itemize}

\end{remark}

\begin{remark}

Consider the largest chain of prime face ideals $\fp_0 \subseteq
\fp_1 \subseteq \cdots \subseteq \fp_n$ in the small support of a
local cohomology module $H_I^r(R)$. In these best case scenario we
have a version of property $2)$ that we introduced at the beginning
of this Section that reads off as:

\vskip 2mm

$\cdot$ $\mu_0(\fp_0, H^r_I(R))= 1$ and  $\mu_j(\fp_0, H^r_I(R))=0$
$\forall j>0$.

\vskip 2mm

$\cdot$ $\mu_i(\fp_i, H^r_I(R))\neq 0$ and $\mu_j(\fp_i,
H^r_I(R))=0$ $\forall j>i$, for all $i=1,...,n$.

\vskip 2mm

\noindent Then:

\vskip 2mm

\begin{itemize}

\item[i)] $ {\rm id}_R H_I^r(R)= {\dim}_R
(\supp_R H_I^r(R))$ if and only if this version of property $2)$ is satisfied.

\vskip 2mm

\item[ii)] $ {\rm id}_R H_I^r(R) = {\dim}_R H_I^r(R)$ if and only if this version ofproperty $2)$ is satisfied
and $\fM \in \supp_R H_I^r(R)$.

\end{itemize}

\vskip 2mm

\noindent This sheds some light on the examples treated in
\cite{He08} where the question whether the equality $ {\rm id}_R
H_I^r(R) = {\dim}_R H_I^r(R)$ holds is considered.
On the other end of possible cases we may have:

\vskip 2mm

$\cdot$ $\mu_0(\fp_0, H^r_I(R))=\mu_0(\fp_n, H^r_I(R)) = 1$ and
$\mu_j(\fp_0, H^r_I(R))=\mu_j(\fp_n, H^r_I(R))=0$ $\forall j>0$.

\vskip 2mm

\noindent Notice that in this case the same property holds for any
prime ideal $\fp_i$ in the chain. In particular all the primes in
the chain are associated primes of $H^r_I(R)$.

\vskip 2mm

\end{remark}





\section{Matlis dual of local cohomology modules}

The minimal projective resolution of a regular holonomic
$D_{R|k}$-module with variation zero $M$ is in the
form:$$\mathbb{P}^{\bullet}(\cM): \hskip 3mm \xymatrix{ \cdots
\ar[r]^{d^m}& P^{m} \ar[r]^{d^{m-1}}& \cdots \ar[r]^{d^1}& P^{1}
\ar[r]^{d^0}&  P^0 \ar[r]& 0},$$ where the $j$-th term is $$P^j =
\bigoplus_{\alpha\in\{0,1\}^n} {R_{{\bf
x}^{\alpha}}}^{\pi_{j}(\fp_\alpha,M)}.$$ The dual {Bass numbers} of
$M$ with respect to the face ideal $\fp_\alpha \subseteq R$ are the
invariants defined by $\pi_{j}(\fp_\alpha,M)$. These invariants can
be computed using the following form of Matlis duality introduced in
\cite{Al05}:
$$M^{\ast}:=\Hom_{D_{R|k}}(M,E_{\bf 1})$$

This is a shift by ${\bf 1}$ of the usual Matlis duality of
$\bZ^n$-graded modules but it has the advantage of being a duality
in the lattice $\{0,1\}^n$, i.e. is a duality of the type $\alpha
\to {\bf 1}-\alpha$ instead of a duality of the type $\alpha \to
-\alpha$ among its graded pieces. In particular, the  $n$-hypercube
$\cM^{\ast}$ corresponding to $M^{\ast}$ satisfy:

\begin{itemize}

 \item [$\cdot$] $\cM^{\ast}_\alpha = \cM_{\mathbf{1}-\alpha}$

\vskip 2mm

\item [$\cdot$] The map $u^{\ast}_{\alpha,i}:\cM^{\ast}_{\alpha}\lra \cM^{\ast}_{\alpha
+\varepsilon_i}$ is the dual of
$u_{\mathbf{1}-\alpha-\varepsilon_i,i}:\cM_{\mathbf{1}-\alpha-\varepsilon_i}\lra
\cM_{\mathbf{1}-\alpha}$.
\end{itemize}

It is easy to check out that the Matlis dual of an injective
$D^T_{v=0}$-module is projective, more precisely we have
$E_\alpha^\ast= R_{{\bf x}^{{\bf 1}-\alpha}}$ and the Matlis dual of
a simple $D^T_{v=0}$-module is simple, namely we have
$(H^{|\alpha|}_{\fp_\alpha}(R))^{\ast}=H^{|{\bf
1}-\alpha|}_{\fp_{{\bf 1}-\alpha}}(R).$

\vskip 2mm

\begin{proposition}\cite[Prop. 5.3]{Al05} With the previous notation
 $$\pi_{p}(\fp_\alpha,M):=
\mu_{p}(\fp_{{\bf 1}-\alpha},M^{\ast}).$$

\end{proposition}

To compute the later Bass number we may assume $\fp_{{\bf
1}-\alpha}=\fM$ is the maximal ideal just using localization so it
boils down to compute the homology of the degree {\bf 1} part of the
\v{C}ech complex:

$$[{{\check{C}}}_{\fM}(M^{\ast})]_{\bf
1}^\bullet: 0\longleftarrow [M^{\ast}]_{\bf
1}\stackrel{\overline{d_{0}}}\longleftarrow \bigoplus_{|\alpha|=1}
[M^{\ast}_{{\bf x^{\alpha}}}]_{\bf
1}\stackrel{\overline{d_{1}}}\longleftarrow
\cdots\stackrel{\overline{d_{p-1}}}\longleftarrow
\bigoplus_{|\alpha|=p} [M^{\ast}_{{\bf x^{\alpha}}}]_{\bf
1}\stackrel{\overline{d_{p}}}\longleftarrow \cdots
\stackrel{\overline{d_{n-1}}}\longleftarrow [M^{\ast}_{{\bf x}^{\bf
1}}]_{\bf 1}\longleftarrow 0$$

\noindent On the other hand, we can also construct the following
complex of $k$-vector spaces from the $n$-hypercube associated to
$M$:
$${{\cM}^{\ast}}^\bullet: 0\longleftarrow [M]_{\bf 0}\stackrel{u^{\ast}_{0}}\longleftarrow
\bigoplus_{|\alpha|=1}
[M]_{\alpha}\stackrel{u^{\ast}_{1}}\longleftarrow
\cdots\stackrel{u^{\ast}_{p-1}}\longleftarrow \bigoplus_{|\alpha|=p}
[M]_{\alpha}\stackrel{u^{\ast}_{p}} \longleftarrow \cdots
\stackrel{u^{\ast}_{n-1}}\longleftarrow [M]_{\bf 1}\longleftarrow
0$$ where the map between summands $[M]_{\alpha} \lra
[M]_{\alpha-\varepsilon_i}$ is sign$(i,\alpha-\varepsilon_i)$ times
the dual of the canonical map $u_{\alpha-\varepsilon_i,i}$. Namely,
${{\cM}^{\ast}}^\bullet$ is the dual, as $k$-vector spaces, of
${{\cM}}^\bullet$. We can mimic what we did for Bass numbers to
obtain:

\begin{proposition}
Let $M\in D_{v=0}^T$ be a regular holonomic $D_{R|k}$-module with
variation zero and $\cM^{\ast \bullet}$ its corresponding complex
associated to the $n$-hypercube. Then, there is an isomorphism of
complexes ${\cM}^{\ast \bullet}\cong
[{\check{C}}_{\fM}(M^\ast)]_{\bf 1}^\bullet.$

\end{proposition}

Therefore we have the following characterization of Bass numbers:

\begin{corollary}
Let $M\in D_{v=0}^T$ be a regular holonomic $D_{R|k}$-module with
variation zero and $\cM^{\ast \bullet}$ its corresponding complex
associated to the $n$-hypercube. Then $$\pi_{p}({{\fp}_{{ 0}}},M)=
{\rm dim}_k \hskip 1mm {\rm H}_{p}({\cM^{\ast}}^{\bullet})$$

\end{corollary}

\vskip 2mm

\subsection{The local cohomology case}

Let $\cM=\{[H^r_I(R)]_{\alpha}\}_{\alpha\in\{0,1\}^n} $ be the
$n$-hypercube of a local cohomology module $H^r_I(R)$ supported on a
monomial ideal. As in the case of Lyubeznik numbers we can also
relate dual Bass numbers of $\fp_0$ to  the $r$-linear
strand of the Alexander dual ideal $I^\vee$. In this case, if
$$\mathbb{L}_{\bullet}^{<r>}(I^{\vee}): \hskip 3mm \xymatrix{ 0
\ar[r]& L_{n-r}^{<r>} \ar[r]& \cdots \ar[r]& L_1^{<r>} \ar[r]&
L_{0}^{<r>} \ar[r]& 0},$$ is the $r$-linear strand of  $I^\vee$
then we consider its monomial matrices to obtain a complex of $k$-vector
spaces indexed as follows:
$$\mathbb{F}_{\bullet}^{<r>}(I^{\vee}): \hskip 3mm \xymatrix{ 0 \ar[r]&
K_{n}^{<r>} \ar[r]& \cdots \ar[r]& K_{r+1}^{<r>} \ar[r]& K_{r}^{<r>}
\ar[r]& 0 }$$ Therefore we obtain analogous results to those in
Section $4.1$

\vskip 2mm

\begin{proposition}
The complex of $k$-vector spaces $\cM^{\ast \bullet}$  associated to
the $n$-hypercube of a fixed local cohomology module $H^r_I(R)$ is
isomorphic to the complex $\mathbb{F}_{\bullet}^{<r>}(I^{\vee})$
obtained from the $r$-linear strand
$\mathbb{L}_{\bullet}^{<r>}(I^{\vee})$ of the Alexander dual ideal
of $I$.
\end{proposition}

\vskip 2mm

\begin{corollary}
Let $\mathbb{F}_{\bullet}^{<r>}(I^{\vee})$ be the complex of $k$-vector spaces
obtained from the $r$-linear strand of the minimal free resolution of the
Alexander dual ideal $I^{\vee}$. Then
$$\pi_{p}(\fp_0,H_I^{r}(R))= {\rm dim}_k
H_{p}(\mathbb{F}_{\bullet}^{<r>}(I^{\vee}))$$
\end{corollary}

\begin{example}
Consider the ideal $I=(x_1,x_4)\cap (x_2,x_5)\cap  (x_1,x_2,x_3)$ in
Example $5.7$. The non-vanishing pieces of the hypercube associated
to the Matlis dual of the corresponding local cohomology modules
are:

\vskip 2mm

 $[(H_{I}^2(R))^{\ast}]_\alpha= k$  \hskip 2mm for $\alpha=(0,1,1,0,1),(1,0,1,1,0)$.

 $[(H_{I}^3(R))^{\ast}]_\alpha= k$  \hskip 2mm for $\alpha=(0,0,0,1,1),(0,0,0,0,1),(0,0,0,1,0),(0,0,1,0,0),(0,0,0,0,0)$.

\vskip 2mm

\noindent In this case we have $(H_{I}^2(R))^{\ast} \cong H_{(x_2,x_3,x_5)}^3(R) \oplus H_{(x_1,x_3,x_4)}^3(R)$ and
the complex associated to the hypercube of  $(H_{I}^3(R))^{\ast}$ is:
$${\xymatrix { 0 & k \ar[l]  && k^3 \ar[ll]_{{\Tiny \begin{pmatrix}
 -1 & 1 & 0  \end{pmatrix}}} && k \ar[ll]_{{\Tiny
\left(\begin{array}{rrr}
1\\ 1\\ 1  \end{array}\right)}}
 & 0
\ar[l] }}$$



\vskip 2mm

\noindent Then, the dual Bass numbers are:

\vskip 2mm

$\pi_0(\fp_\alpha,H_{I}^2(R))=1 $ \hskip 2mm for
$\alpha=(1,0,0,1,0),(0,1,0,0,1)$.

$\pi_1(\fp_\alpha,H_{I}^2(R))=1 $ \hskip 2mm for
$\alpha=(1,0,0,0,0),(0,0,0,1,0),(0,1,0,0,0),(0,0,0,0,1)$.

$\pi_2(\fp_\alpha,H_{I}^2(R))=2 $ \hskip 2mm for
$\alpha=(0,0,0,0,0)$.

\vskip 3mm

$\pi_0(\fp_\alpha,H_{I}^3(R))=1 $ \hskip 2mm for
$\alpha=(1,1,1,1,1)$.

$\pi_1(\fp_\alpha,H_{I}^3(R))=1 $ \hskip 2mm for
$\alpha=(1,1,0,0,1),(1,1,0,1,0),(0,1,1,1,1),(1,0,1,1,1)$.

$\pi_2(\fp_\alpha,H_{I}^3(R))=1 $ \hskip 2mm for
$\alpha=(1,1,0,0,0),(1,0,0,1,0),(1,0,0,0,1),(0,1,0,1,0),(0,1,0,0,1),$

\hskip 5cm $(0,0,1,1,1)$.

$\pi_3(\fp_\alpha,H_{I}^3(R))=1 $ \hskip 2mm for
$\alpha=(1,0,0,0,0),(0,1,0,0,0),(0,0,0,1,0),(0,0,0,0,1)$.

$\pi_4(\fp_\alpha,H_{I}^3(R))=1 $ \hskip 2mm for
$\alpha=(0,0,0,0,0)$.

\vskip 3mm

\begin{remark}
It is worth to point out that one may find examples of modules
having the same Bass numbers but different dual Bass numbers. For
example, consider $H_{I}^3(R)$ in the previous example and
$H_{(x_1,x_2,x_4,x_5)}^4(R) \oplus E_{(1,1,1,0,0)}$. The
non-vanishing parts of their corresponding hypercubes  are
respectively
{\tiny $${\xymatrix {  k \ar[d]_{1} \ar[dr]^(.3){1}  & & \\
k\ar[dr]_{1}& k \ar[d]_{1} & k \ar[dl]^{1}
\\& k & }} \hskip 1cm
{\xymatrix {  k \ar[d]_{1} \ar[dr]^(.3){1}  & & \\ k\ar[dr]_{1}& k
\ar[d]_{1} & k \ar[dl]^{0}
\\& k & }}
$$} \noindent Notice that these modules are not isomorphic.

\end{remark}

\end{example}

\section{Questions}

The approach we take in this work to study Lyubeznik numbers opens
up a number of questions just because of its relation to free
resolutions. We name here a few and we hope that they will be
addressed elsewhere.

\vskip 2mm

$\bullet$ {\bf Topological description of Lyubeznik numbers: } 
A recurrent topic in recent years has been to attach a cellular
structure to the free resolution of a monomial ideal. In general
this can not be done as it is proved in \cite{Vel08} but there are
large families of ideals having a cellular resolution.  Using the dictionary described
in Section $4$ we can translate the same questions to Lyubeznik
numbers. In particular we would be interested in finding cellular
structures on the linear strands of a free resolution so one can
give a topological description of Lyubeznik numbers.

\vskip 2mm

Another question that immediately pops up is the behavior of
Lyubeznik numbers with respect to the characteristic of the field.
In \cite{DaKu} it is proved that the Betti table in characteristic
zero is obtained from the positive characteristic Betti table by a
sequence of consecutive cancelations, i.e. cancelation of terms in
two different linear strands as it can be seen in Example $4.6$. In
our situation not only the cancelation affects the behavior, we also
have to put into the picture the acyclicity of the linear strands.

\vskip 2mm

We do not know whether it is possible to find an example where the Betti table depends on
the characteristic but the Lyubeznik table does not. Such an example
would require that the Betti table in characteristic zero is
obtained from the positive characteristic Betti table by at least
two consecutive cancelations.

\vskip 2mm

$\bullet$ {\bf Injective resolution of local cohomology modules: }
When $R/I$ is Cohen-Macaulay, a complete description of the
injective resolution of $H_I^r(R)$, i.e. Bass numbers and maps
between injective modules, was given in \cite{Ya01}. The question
on how to find a general description for any ideal might be too difficult
so we turn our attention to some nice properties of the resolution.

\vskip 2mm

The injective resolution can be decomposed in linear strands. Namely, if
$$\mathbb{I}_{\bullet}(M): \hskip 3mm \xymatrix{ 0 \ar[r]& M \ar[r] & I_0  \ar[r]& I_1 \ar[r]& \cdots \ar[r]& I_m
\ar[r]& 0}$$ is the injective resolution of a module with variation zero then, given an integer $r$, the {\it
$r$-linear strand} of $\mathbb{I}_{\bullet}(M)$ is the complex:
$$\mathbb{I}_{\bullet}^{<r>}(M): \hskip 3mm \xymatrix{ 0 \ar[r]&
I_{0}^{<r>} \ar[r]& I_{1}^{<r>} \ar[r]& \cdots
\ar[r]& I_{m}^{<r>} \ar[r]& 0},$$ where $$I_j^{<r>} =
\bigoplus_{|\alpha|=j+r} E_\alpha^{\mu_{j}(\fp_\alpha, M)},$$

\vskip 2mm

When $R/I$ is Cohen-Macaulay, we have $\mu_{p}(\fp_\alpha, H_I^{\hlt
I}(R))= \delta_{p,n-|\alpha|}$ for all face ideals in the support of
$R/I$, so the injective resolution of $H_I^{\hlt I}(R)$ behaves like
the injective resolution of a Gorenstein ring. In particular, this
resolution is linear.  When we turn our attention to minimal
non-Cohen-Macaulay we see that the injective resolution of
$H_I^{\hlt I}(R)$ behaves like that of a Gorenstein ring except for
the Bass number with respect to the maximal ideal. Notice that the
module $H^2_{\fa_4}(R)$ in Example $4.4$ has a $2$-linear injective
resolution with $\mu_2(\fM,H^2_{\fa_4}(R))=2$ but the resolution of
$H^2_{\fa_5}(R)$ has two linear strands since
$\mu_2(\fM,H^2_{\fa_5}(R))=\mu_3(\fM,H^2_{\fa_5}(R))=1$. As in the
case of free resolutions it would be interesting to study the
different linear strands in the injective resolution of $H^r_I(R)$
and how these linear strands depend on the other local cohomology
modules $H^s_I(R)$, $s\neq r$.

\vskip 2mm $\bullet$ {\bf Projective resolution of local cohomology
modules: } The same questions we posted above for injective
resolutions can be asked for projective resolutions. We have to
point out that F.~Barkats \cite{Ba95} gave an algorithm to compute a
presentation of the local cohomology modules $H^r_I(R)$ using in an
implicit way a projective resolution of these modules with variation
zero. However she was only able to compute effectively examples in
the polynomial ring $k[x_1,...,x_6]$.

\end{document}